\documentclass{amsart}

\usepackage{latexsym,amssymb,amsmath,enumitem,verbatim,todonotes,xurl,amsthm,nccmath,mathtools,mathrsfs}
\usepackage{thmtools}
\usepackage[hidelinks]{hyperref}
\usepackage[capitalise,nameinlink]{cleveref}

\makeatletter
\@namedef{subjclassname@2020}{
\textup{2020} Mathematics Subject Classification}
\makeatother

\title[Class Symmetric Systems]{A General Theory of Class Symmetric Systems}

\DeclareMathOperator{\Ord}{Ord}

\DeclareMathOperator{\dom}{dom}
\DeclareMathOperator{\id}{id}

\DeclareMathOperator{\1}{\mathsf{1}}

\DeclareMathOperator{\sym}{sym}
\DeclareMathOperator{\fix}{fix}
\DeclareMathOperator{\R}{\textit{R}}
\DeclareMathOperator{\forces}{\Vdash}
\DeclareMathOperator{\witforces}{\forces_{\!\!\ast}}
\DeclareMathOperator{\atforces}{\forces_0}
\DeclareMathOperator{\Sforces}{\forces_{\!\SSS}}
\DeclareMathOperator{\Pforces}{\forces_{\PP}}
\DeclareMathOperator{\notatforces}{\not{\!\forces}_0}
\DeclareMathOperator{\notSforces}{\not{\!\forces}_{\!\SSS}}
\DeclareMathOperator{\sforces}{\forces_{\!\SSS}^\ast}
\DeclareMathOperator{\pforces}{\forces_{\PP}^\ast}
\DeclareMathOperator{\Gforces}{\forces_{\!\SSS}^{\!\vec\Gamma}}

\newcommand{\ZFC}{\mathrm{ZFC}}
\newcommand{\ZF}{\mathrm{ZF}}
\newcommand{\GB}{\mathrm{GB}}

\newcommand{\PP}{\mathbb P}
\newcommand{\QQ}{\mathbb Q}

\newcommand{\C}{\mathcal C}
\newcommand{\G}{\mathcal G}
\newcommand{\F}{\mathcal F}
\newcommand{\SSS}{\mathcal S}
\newcommand{\MM}{\mathcal M}

\newcommand{\HS}{\mathrm{HS}}
\newcommand{\MMGS}{\MM[G]_\SSS}
\newcommand{\ifff}{\leftrightarrow}
\newcommand{\iffq}{\quad\iff\quad}
\newcommand{\concatt}{\mathbin{\raisebox{1ex}{\scalebox{.7}{$\!\frown\!\!$}}}}
\newcommand{\vecdot}[1]{\vec{\dot #1}}

\def\<#1>{\langle#1\rangle}
\newbox\gnBoxA\newdimen\gnCornerHgt\setbox\gnBoxA=\hbox{$\ulcorner$}\global\gnCornerHgt=\ht\gnBoxA\newdimen\gnArgHgt\def\Godelnum#1{\setbox\gnBoxA=\hbox{$#1$}\gnArgHgt=\ht\gnBoxA\ifnum\gnArgHgt<\gnCornerHgt\gnArgHgt=0pt\else\advance\gnArgHgt by -\gnCornerHgt\fi\raise\gnArgHgt\hbox{$\ulcorner$}\box\gnBoxA\raise\gnArgHgt\hbox{$\urcorner$}}

\theoremstyle{definition}
\newtheorem{fact}{Fact}
\newtheorem{lemma}[fact]{Lemma}
\newtheorem{theorem}[fact]{Theorem}
\newtheorem{corollary}[fact]{Corollary}
\newtheorem{claim}[fact]{Claim}

\newtheorem{observation}[fact]{Observation}
\newtheorem{proposition}[fact]{Proposition}

\newtheorem{question}[fact]{Question}

\newtheorem{definition}[fact]{Definition}

\author{Peter Holy}
\address{Institut f\"ur diskrete Mathematik und Geometrie\\
TU Wien\\
Wiedner Hauptstrasse 8-10/104\\
1040 Vienna\\
Austria}
\email{peter.holy@tuwien.ac.at}

\author{Emma Palmer}
\address{Mathematical Institute,
University of Oxford,
Andrew Wiles Building,
Radcliffe Observatory Quarter (550),
Woodstock Road,
Oxford,
OX2 6GG,
United Kingdom\\
ORCiD: 0000-0003-0001-5029}
\email{emma.palmer@hertford.ox.ac.uk}

\author{Jonathan Schilhan}
\address{University of Vienna,
Institute of Mathematics,
Kurt G\"odel Research Center,
Kolingasse 14-16,
1090 Vienna,
Austria\\
ORCiD: 0000-0001-6696-1603}
\email{jonathan.schilhan@univie.ac.at}


\thanks{}

\begin{document}

\begin{abstract}
We develop a general theory for class-sized symmetric systems as a natural extension of symmetric systems with respect to class forcing. In particular, adapting the usual notions of pretameness and tameness for class forcing, we present sufficient conditions for the preservation of the axioms of G\"odel-Bernays set theory (without the axiom of choice), and for the forcing theorem to hold for class-sized symmetric systems.
\end{abstract}

\maketitle

\newcommand{\chkfront}{\scalebox{1.6}[1.0]{$\vee$}}

\section{Introduction and Basic Definitions}

Class-sized symmetric systems have been used to verify a number of important set theoretic results, for instance in the construction of the Gitik model \cite{gitik1980}, in which all uncountable cardinals are singular. Other important examples are given by the Morris model \cite{Morris} (see also \cite{Karagila2020}) or the Bristol model \cite{bristol}, \cite{KARAGILA2026}. The basic properties of these symmetric systems -- whether they satisfy the forcing theorem and preserve the relevant axioms -- have been established on an ad hoc case-by-case basis for these particular systems. The main point of our paper is to provide a general theory of class-sized symmetric systems: we provide useful sufficient conditions, which are based around the well-known notions of pretameness and tameness for class forcing, for when class-sized symmetric systems admit forcing relations and satisfy the truth lemma (these properties are commonly jointly referred to as the \emph{forcing theorem}), and for when they preserve the axioms of G\"odel-Bernays set theory (except for the axiom of choice).
Our results also provide a possible answer to the second question in \cite[Question 7.14]{MATTHEWSpaper}, where Richard Matthews asks for a combinatorial condition on class-sized symmetric systems that ensures the preservation of the axioms of G\"odel-Bernays set theory other than the axiom of choice and the powerset axiom.

\bigskip

We denote G\"odel-Bernays set theory without the axiom of choice by $\GB$, and unless otherwise noted, this will be our base theory in this paper. We present $\GB$ as a two-sorted theory\footnote{Alternatively, $\GB$ can be viewed as a one-sorted theory, where the sets are exactly those classes which are elements of other classes.}, that talks about a domain of sets and a domain of classes; we follow the convention that set variables are given by lowercase letters $x,y,a,b\ldots$ while class variables are given by uppercase letters $X,Y,A,B\dots$.

\medskip

The axioms of $\GB$ are:
\begin{itemize}
\item The following set axioms: Extensionality, Regularity, Pairing, Union, Infinity, and Powerset.
\item Extensionality for classes: for any classes $A,B$, if $(\forall x)(x\in A\ifff x\in B)$, then $A=B$.
\item Every set represents a class: for any set $a$, there is a class $A$ such that $(\forall x)(x\in a\;\ifff\;x\in A)$.
\item Separation and Collection, articulated as \emph{single} axioms:
\begin{itemize}
\item For any set $a$ and class $A$, there is a set $b$ which is their intersection: $(\forall x)(x\in b\;\ifff\;x\in a\land x\in A)$.
\item For any set $a$ and class relation $R$ such that $(\forall x\in a)(\exists y)(\<x,y>\in R)$, there is a set $b$ such that $(\forall x\in a)(\exists y\in b)(\<x,y>\in R)$.
\end{itemize}
\item The Class Existence axioms: Membership, Intersection, Complement, Domain, Product by $V$, Circular permutation, and Transposition\footnote{Some examples: Product by $V$ is $(\forall X\exists Y)(\forall x)(x\in Y\ifff(\exists y\exists z)(x=\<y,z>\land y\in X))$ and Transposition is $(\forall X\exists Y)(\forall x\forall y\forall z)(\<x,y,z>\in Y\ifff\<x,z,y>\in X)$.}. Together, these are equivalent with first-order class comprehension (over Extensionality, Regularity, and Pairing): given a first-order formula\footnote{I.e., a formula with no second-order quantifiers. In this paper, first-order formulas can include atomic formulas of the form $x\in X$ and $X=Y$ (where $X,Y$ are class variables) as sub-formulas.} $\varphi$ with class parameters $\vec A$, there is a class $B$ satisfying $(\forall x)(x\in B\;\ifff\;\varphi(x,\vec A))$, which we write as $B=\{x\mid\varphi(x,\vec A)\}$ (cf. \cite[Ch.~2]{godel1940consistency}).
\end{itemize}

Notably, this is a finite axiomatization of $\GB$. $\GB^-$ denotes $\GB$ without the powerset axiom; as we can see from above, this also has a finite axiomatization. We will sometimes additionally assume the second-order class comprehension scheme: given any second-order formula $\varphi$ with class parameters $\vec A$, there is a class ${B={\{x\mid\varphi(x,\vec A)\}}}$ (together with $\GB$ and global choice, this comprises Kelley-Morse set theory). However, unless otherwise specified, from here we assume only $\GB$.

If $I$ is a class, we say that a class $D$ codes a sequence $\<D_i\mid i\in I>$ of classes $D_i$ if $D=\{\<i,d>\mid i\in I\;\land\;d\in D_i\}$, and with a \emph{sequence of classes}, we always mean such a sequence that is coded by a single class. We will sometimes not distinguish between sequences of classes and their codes (even though only the latter actually exist as classes). Analogous codings and remarks apply to ordered tuples of classes.

\begin{definition}\label{def:system}
A \emph{class symmetric system} is a triple of the form $\SSS=\<\PP,\G,\F>$ with the following properties:
\begin{itemize}
\item $\PP=\<P,\le>$ is a preorder, where $P,\le$ are classes. We also say that $\PP$ is a \emph{class forcing notion} in this case. We use $\1$ to denote its weakest element.
\item $\G$ is a \emph{group of automorphisms of $\PP$, parametrized by sets}; that is, there is a class $C_\G \subseteq V\times P^2$ so that for each $a\in\dom C_\G$, the class $\pi_a={\{\<p,q>\mid\<a,p,q>\in C_\G\}}$ is an automorphism of $\mathbb{P}$ and $\G=\{\pi_a\mid a\in\dom C_\G\}$ forms a group.\footnote{Note that while $\G$ is formally a third-order object, the corresponding property of $C_\G$ can be expressed by a first-order formula using only the classes $P$ and $\leq$ as parameters.} We identify $C_\G$ with $\G$.
\item $\F$ is a \emph{normal filter of subgroups of $\G$, with a base that is parametrized by sets}; that is, there is a class $C_\F\subseteq V\times\dom C_\G$ so that for every $a\in\dom C_\F$,
\begin{itemize}
\item $H_a=\{\pi_d\mid\<a,d>\in C_\F\}$ forms a subgroup of $\G$,
\item $(\exists b\in\dom C_\F)(H_b=\G)$,
\item for every $b\in\dom C_\F$, there is $c\in\dom C_\F$ so that $H_c\subseteq H_a\cap H_{b}$,
\item for every $d\in\dom C_\G$, there is $c\in\dom C_\F$ so that $H_c\subseteq\pi_d H_a\pi_d^{-1}$.\footnote{Again, this all can be expressed in a first-order way.}
\end{itemize}
$\F$ consists of all those subgroups $H$ of $\G$ of the form $\{\pi_b\mid b\in C\}$, with $C$ a class, for which there is $a\in\dom C_\F$ with $H_a\subseteq H$.\footnote{Note that $\F$ is formally a fourth-order object.} We identify $C_\F$ with~$\F$.
\end{itemize}
\end{definition}

Any statement we will want to make about $\G$ and $\F$ will be first-order expressible. For instance, all of the following: 

\begin{itemize}
\item $\pi\in\G$, for any class $\pi$, 
\item $\pi_a\pi_b=\pi_c$, $\pi_a=\pi_b^{-1}$, $\pi_a=\id$,
\item $\pi_a\in H_b$, $\pi_a H_b=H_c$, $H_b\pi_a=H_c$, 
\item $\pi_a\in H$, $H=H_b$, $H\in\F$, where $H=\{\pi_c\mid c\in C\}$ for some class $C$.
\end{itemize}

\begin{definition}
Let $\SSS=\<\PP,\G,\F>$ be a (class) symmetric system.
\begin{itemize}
\item A set or class $A\subseteq P$ is \emph{symmetric} in case \[\sym(A):=\{\pi\in\G\mid\pi[A]=A\}\in\F.\]
\item A set or class $D\subseteq P$ is \emph{symmetrically dense} in case it is symmetric and dense in $\PP$.
\item A \emph{$\PP$-name} is some set or class $\dot x$ with elements that are ordered pairs of the form $\<\dot y,p>$ with (recursively) a (set-size) $\PP$-name $\dot y$ and $p\in P$.
\item $V^{\PP}$ is the class of all set-size $\PP$-names.
\item A $\PP$-name $\dot x$ is \emph{symmetric} in case \[\{\pi\in\G\mid\pi(\dot x)=\dot x\}\in\F,\textrm{ where}\] \[\pi(\dot x)=\{\<\pi(\dot y),\pi(p)>\mid\<\dot y,p>\in\dot x\}\textrm{ recursively}.\]
\item A symmetric $\PP$-name $\dot x$ is \emph{hereditarily symmetric} in case (recursively) every $\dot y\in\dom(\dot x)$ is hereditarily symmetric; we also call such names $\SSS$-names.
\item $\HS$ is the class of all set-size $\SSS$-names.
\item We let $\rho$ denote the usual function that recursively assigns each set-size $\PP$-name its name rank, and for an ordinal $\alpha$, we let $\HS_\alpha$ denote the class of set-size $\SSS$-names $\dot x$ of name rank $\rho(\dot x)$ less than $\alpha$.\footnote{Note that if $\alpha>1$ and $P$ is a proper class, then $\HS_\alpha$ is indeed a proper class.}
\item We will use the following notation: If $X\subseteq V^{\PP}$, then $X^\bullet=\{\<\dot x,\1>\mid\dot x\in X\}$.
\end{itemize}
\end{definition}

For most applications, we will usually want our symmetric extensions to satisfy the axioms of at least $\GB$.
In the case of (standard) class forcing, there are two prominent niceness properties of class forcing: 

\begin{definition}
A notion of class forcing $\PP$ is \emph{pretame} if for every $p\in P$ and every sequence $\<D_i\mid i\in I>$ of dense subclasses of~$\PP$ with $I$ a set, there is $q\le p$ and a set $\<d_i\mid i\in I>$ such that $d_i\subseteq D_i$ and $d_i$ is predense below $q$ whenever $i\in I$.
\end{definition}

Pretameness is equivalent to the existence of forcing relations (see \cref{def:fr}) together with the preservation of the axioms of $\GB^-$ (see \cite{hks}). The second important property is \emph{tameness}, which is pretameness together with the requirement that the powerset axiom is preserved.\footnote{As in \cite{syhandbook}, this can be expressed as a simple forcing statement, given that pretameness already implies the existence of forcing relations; this property also has a combinatorial characterization, see \cite[Definition 6.2]{fh}, or \cite[Section 2.2]{friedman} for the original version.}
A main goal of this paper is to investigate analogous notions for class symmetric systems.

\section{Atomic Forcing Relations}\label{section:forcingtheorem}

In this section, we define forcing relations for atomic formulas (\cref{def:atomic}) and give a combinatorial condition on class symmetric systems (\cref{def:compt}) which ensures that a system admits forcing relations for atomic formulas (\cref{lem:atomic}).

We continue to work in the theory $\GB$. In the following, let $\SSS=\<\PP,\G,\F>$ always denote a class symmetric system.

\begin{definition}\label{def:atomic}
A class relation $\atforces$ is a \emph{forcing relation for atomic formulas} (for $\SSS$) if $\atforces$ satisfies the following properties for all $\dot x,\dot y\in\HS$:\footnote{We write $p\atforces\dot x R\dot y$ rather than $\<p,\dot x,R,\dot y>\in\atforces$, for $R\in\{\in,\subseteq,=\}$. In this section, we consider $xRy$ for $R\in\{\in,\subseteq,=\}$ and variables $x$ and $y$ to be our atomic formulas.}
\begin{enumerate}
\item $p\atforces\dot x\in\dot y$ if and only if the class
\[
\{q\in P\mid(\exists\<\dot z,r>\in\dot y)(q\leq r\,\land\,q\atforces\dot x=\dot z)\}
\]
is dense below $p$;\label{ad:el}
\item $p\atforces\dot x\subseteq\dot y$ if and only if for any $\<\dot z,r>\in\dot x$, the class
\[
\{q\in P\mid q\atforces\dot z\in\dot y\}
\]
is dense below $p\wedge r$;\footnote{By this we do not mean to imply that there exists a meet $p\wedge r\in P$, but rather that for any condition $s$ below both $p$ and $r$, there is $q$ in the given class such that $q\leq s$.}\label{ad:su}
\item $p\atforces\dot x=\dot y$ if and only if $p\atforces\dot x\subseteq\dot y$ and $p\atforces\dot y\subseteq\dot x$.\label{ad:eq}
\end{enumerate}
\end{definition}

When such a class exists, we say that $\SSS$ \emph{admits} a forcing relation for atomic formulas. Indeed, in that case, it is uniquely determined on atomic formulas, as is immediate by the usual induction on name rank:

\begin{lemma}\label{lem:uniqueat}
Any two forcing relations $\atforces,\atforces'$ for atomic formulas (for $\SSS$) will agree on all tuples $\<p,\dot x,R,\dot y>$ for $p\in P$, $R\in\{\in,\subseteq,=\}$, and $\dot x,\dot y\in\HS$.\hfill{$\qed$}
\end{lemma}

A forcing relation for atomic formulas also satisfies some familiar properties of forcing relations.

\begin{lemma}\label{lem:atsymm}
Any forcing relation $\atforces$ for atomic formulas (for $\SSS$) satisfies the symmetry lemma: for $\pi\in\G$, $p\in P$, $\R\in\{\in,\subseteq,=\}$, and $\dot x,\dot y\in\HS$,
\[
p\atforces\dot x\R\dot y\iffq\pi(p)\atforces\pi(\dot x)\R\pi(\dot y).
\]
\end{lemma}

\begin{proof}
This is shown in exactly the same way as for usual (set-sized) symmetric systems: by induction on the name rank of $\dot x,\dot y\in\HS$, using the equivalences in \cref{def:atomic}. For example, if $p\atforces\dot x\in\dot y$, then
\[
\{q\in P\mid(\exists\<\dot z,r>\in\dot y)(q\leq r\,\land\,q\atforces\dot x=\dot z)\}
\]
is dense below $p$, so by the inductive hypothesis,
\[
\{\pi(q)\mid q\in P\,\land\,(\exists\<\pi(\dot z),\pi(r)>\in\pi(\dot y))(\pi(q)\leq\pi(r)\,\land\,\pi(q)\atforces\pi(\dot x)=\pi(\dot z))\}
\]
is dense below $\pi(p)$, proving that $\pi(p)\atforces\pi(\dot x)\in\pi(\dot y)$.
\end{proof}

\begin{lemma}\label{lem:atom}
Suppose that $\SSS$ admits a forcing relation $\atforces$ for atomic formulas.
\begin{enumerate}
\item If $\<\dot x,p>\in\dot y$, then $p\atforces\dot x\in\dot y$.
\item For all $p\in P$,\; $p\atforces\dot x=\dot x$.
\end{enumerate}
\end{lemma}

\begin{proof}
Considering \cref{def:atomic}, it is straightforward to show (1) and (2) simultaneously by induction on name rank.
\end{proof}

When does $\SSS$ admit a forcing relation for atomic formulas? Given second-order class comprehension, the answer is always.

\begin{lemma}\label{lem:soatomic}
Assume second-order class comprehension\footnote{Note that only finitely many instances of the second-order class comprehension scheme are actually needed for the proof, so this result could be restated as a theorem of $\GB$.}. Then $\SSS$ admits a forcing relation for atomic formulas.
\end{lemma}

Roughly, this follows because second-order class comprehension allows us to do recursion on classes. See \cite[Lemma~15]{Antos2018} for a proof in the context of usual class forcing; applying this to $\PP$ and then restricting the result to $\HS$ gives us a forcing relation for atomic formulas for $\SSS$.

However, when we only assume $\GB$, this argument no longer works (and in fact there can be class forcing notions that do not admit forcing relations for atomic formulas, cf. \cite[Corollary~7.6]{hklns}). Instead, we give a condition on $\SSS$, based on the idea of pretameness, which ends up ensuring that $\SSS$ admits forcing relations for atomic formulas.

\begin{definition}\label{def:compt}
\begin{itemize}
\item A sequence $\<D_i\mid i\in I>$ of subclasses of $\PP$  is \emph{symmetric} if $I\subseteq\HS$ is such that $I^\bullet$ is symmetric and there is $H\in\F$ such that $\pi[D_i]=D_{\pi(i)}$ whenever $\pi\in H$ and $i\in I$.
\item For $q\in P$, a class $D\subseteq P$ is \emph{predense below $q$} in case any $r\le q$ is compatible with some element of $D$.
\item A class symmetric system $\SSS=\<\PP,\G,\F>$ is \emph{combinatorially pretame} if for every $p\in P$ and every symmetric sequence $\<D_i\mid i\in I>$ of dense subclasses of~$\PP$  with $I$ a set, there is $q\le p$ and a set $\<d_i\mid i\in I>$ such that $d_i\subseteq D_i$ and $d_i$ is predense below $q$ whenever $i\in I$.
\end{itemize}
\end{definition}

Note that if $\<D_i\mid i\in I>$ is a symmetric sequence of subclasses of $\PP$, as witnessed by $H\in\F$, then $\pi[D_i]=D_i$ whenever $i\in I$ and $\pi\in H\cap\sym(i)$, hence every $D_i$ is symmetric.
Note also that if $\PP$ is pretame, then $\SSS$ is immediately seen to be combinatorially pretame. 

\medskip

Adapting the arguments of M.\ Stanley presented in \cite[Theorem 2.4]{hks}, we show:

\begin{theorem}\label{lem:atomic}
If $\SSS$ is combinatorially pretame, then $\SSS$ admits a forcing relation $\atforces$ for atomic formulas.
\end{theorem}

\begin{proof}
We start by defining an auxiliary relation $\witforces$. For a condition $p\in P$, a relation symbol $\R\in\{\in,\subseteq\}$, and $\dot x,\dot y\in\HS$, we define $p\witforces\dot x\R\dot y$ to mean there are sets $\mathbf{f}\subseteq P\times \HS\times\{\in,\subseteq\}\times\HS$ and $\vec d=\<d^e\mid e\in\mathbf f>$ such that the following hold:
\begin{itemize}
\item $\<p,\dot x,\R,\dot y>\in\mathbf f$.
\item If $e=\<q,\dot u,\in,\dot v>\in\mathbf f$, then $d^e$ is a subset of $P$ that is predense below $q$ such that for all $r\in d^e$, there is $\<\dot w,s>\in\dot v$ with  $r\leq s$ and \[\<r,\dot u,\subseteq,\dot w>,\<r,\dot w,\subseteq,\dot u>\in\mathbf f.\]
\item If $e=\<q,\dot u,\subseteq,\dot v>\in\mathbf f$, then $d^e=\<d^e_{\dot w,s}\mid\<\dot w,s>\in\dot u>$ is a sequence of subsets $d^e_{\dot w,s}$ of $P$ that are predense below $q\wedge s$\footnote{This means that for any condition $t$ below both $q$ and $s$, there is a condition in $d^e_{\dot w,s}$ that is compatible with $t$.} such that $\<r,\dot w,\in,\dot v>\in\mathbf f$ whenever $r\in d^e_{\dot w,s}$.
\end{itemize}
If the above holds, we say that $\mathbf f$ and $\vec d$ witness that $p\witforces\dot x\R\dot y$. We also let $p\witforces\dot x=\dot y$ abbreviate $(p\witforces\dot x\subseteq\dot y)\land(p\witforces\dot y\subseteq\dot x)$.

We can now define the relation $\atforces$ which we will show satisfies the recursion in \cref{def:atomic}: for $p\in P$, $\R\in\{\in,\subseteq,=\}$, and $\dot x,\dot y\in\HS$, we define that \[p\atforces\dot x\R\dot y\iffq\{q\in P\mid q\witforces\dot x\R\dot y\}\textrm{ is dense below }p.\] Before checking the conditions in \cref{def:atomic}, we prove some claims to make the arguments easier.

\begin{claim}\label{claim:at}
\begin{enumerate}
\item If $\<\dot x,p>\in\dot y$, then $p\witforces\dot x\in\dot y$.
\item If $p\witforces\dot x\R\dot y$ and $q\leq p$, then $q\witforces\dot x\R\dot y$.
\item Suppose that $\<\<p_i,\dot{x}_i,\R_i,\dot{y}_i,\mathbf f_i,\vec{d}_i>\mid i\in I>$ is a set such that for each ${i\in I}$, the sets $\mathbf f_i$ and $\vec{d}_i$ witness that $p_i\witforces\dot{x}_i\R_i\dot{y}_i$. Then there are sets $\mathbf f,\vec d$ witnessing that $p_i\witforces\dot{x}_i\R_i\dot{y}_i$ for all $i\in I$ simultaneously.
\item $\witforces$ and $\atforces$ satisfy the symmetry lemma: for $\pi\in\G$, $p\in P$, ${\R\in\{\in,\subseteq,=\}}$, and $\dot x,\dot y\in\HS$, we have both $p\witforces\dot x\R\dot y\iff\pi(p)\witforces\pi(\dot x)\R\pi(\dot y)$ and $p\atforces\dot x\R\dot y\iff\pi(p)\atforces\pi(\dot x)\R\pi(\dot y)$.\footnote{We cannot use \cref{lem:atsymm} at this point, because we do not yet know that $\atforces$ is a forcing relation for atomic formulas.}
\end{enumerate}
\end{claim}

\begin{proof}
(1) Suppose that $\<\dot x,p>\in\dot y$. Let $\mathbf f$ consist of those $\<q,\dot w,\in,\dot v>$ for which $\<\dot w,q>\in\dot v$ and $\dot v$ is in the transitive closure of $\dot y$ or $\dot v=\dot y$, and also of those $\<q,\dot u,\subseteq,\dot u>$ for which $\<\dot u,q>$ is in the transitive closure of $\dot y$. Let $\vec d=\{d^e\mid e\in \mathbf f\}$ where $d^{\<q,\dot v,\in,\dot w>}=\{q\}$ and $d^{\<q,\dot u,\subseteq,\dot u>}_{\dot w,s}=\{s\}$ for all relevant $\dot u,\dot v,\dot w$ and $s$. Then it is clear from the definition of $\witforces$ that $\mathbf f$ and $\vec d$ witness that $p\witforces \dot x\in\dot y$.

(2) Now suppose that $p\witforces\dot x\R\dot y$, as witnessed by $\mathbf f$ and $\vec d$, and that $q\leq p$. Then the sets $\mathbf f\cup\{\<q,\dot x,\R,\dot y>\}$ and $\vec d\cup\{\<\<q,\dot x,\R,\dot y>,d^{\<p,\dot x,\R,\dot y>}>\}$ witness that $q\witforces\dot x\R\dot y$ (because any set that is predense below $p$ will be predense below $q\leq p$).

(3) Next, suppose the premise of (3) holds. Let $\mathbf f=\bigcup_{i\in I}\mathbf f_i$ and $\vec d={\{d^e\mid e\in \mathbf f\}}$, where $d^{\<q,\dot u,\in,\dot v>}=\bigcup_{i\in I}(d_i)^{\<q,\dot u,\in,\dot v>}$ and $d^{\<q,\dot u,\subseteq,\dot v>}_{\dot w,s}=\bigcup_{i\in I}(d_i)^{\<q,\dot u,\subseteq,\dot v>}_{\dot w,s}$ whenever $\<\dot w,s>\in\dot u$.\footnote{If $e\not\in \mathbf f_i$, we consider $(d_i)^e=\emptyset$ and $(d_i)^e_{\dot w,s}=\emptyset$ for any $\<\dot w,s>$.} That $\mathbf f$ and $\vec d$ are as desired follows from the definition of $\witforces$.

(4) Suppose that $\pi\in\G$, $p\in P$, $\R\in\{\in,\subseteq,=\}$, and $\dot x,\dot y\in\HS$. By definition, we have $p\atforces\dot x\R\dot y$ exactly when the class $\{q\in P\mid q\witforces\dot x\R\dot y\}$ is dense below $p$. Thus, if we can show that $\witforces$ satisfies the symmetry lemma, we are done. To that end, suppose that $\mathbf f$ and $\vec d$ witness that $p\witforces\dot x\R\dot y$. Consider $\pi$ applied to $\mathbf f$ and $\vec d$ in the obvious way:
\[
\pi(\mathbf f)=\{\<\pi(q),\pi(\dot u),\R',\pi(\dot v)>\mid\<q,\dot u,\R',\dot v>\in \mathbf f\},\quad\pi(\vec d)=\{\pi[d^e]\mid e\in \mathbf f\}.
\]
Examining the definition of $\witforces$, it is straightforward to see that $\pi(\mathbf f)$ and $\pi(\vec d)$ witness that $\pi(p)\witforces\pi(\dot x)\R\pi(\dot y)$.
\end{proof}

\noindent We now prove that $\atforces$ satisfies the recursion in \cref{def:atomic}. We start with the reverse direction of \cref{def:atomic}(\ref{ad:el}): assume that
\[
\{q\in P\mid(\exists\<\dot z,r>\in\dot y)(q\leq r\,\land\,q\atforces\dot x=\dot z)\}
\]
is dense below $p$. Our definition of $\atforces$ then implies that
\[
D=\{q\in P\mid(\exists\<\dot z,r>\in\dot y)(q\leq r\,\land\,q\witforces\dot x=\dot z)\}
\]
is dense below $p$. We show that $q\witforces\dot x\in\dot y$ for any $q\in D$. To this end, let $q\in D$, say with $\<\dot z,r>\in\dot y$ such that $q\leq r$ and $q\witforces\dot x=\dot z$. By \cref{claim:at}(1), $r\witforces\dot z\in\dot y$, so \cref{claim:at}(2) gives us $q\witforces\dot z\in\dot y$. We use \cref{claim:at}(3) to find $\mathbf f$ and $\vec d$ witnessing that $q\witforces\dot z\in\dot y$ and $q\witforces\dot x=\dot z$ simultaneously. Then $\mathbf f\cup\{\<q,\dot x,\in,\dot y>\}$ and $\vec d\cup\{\<\<q,\dot x,\in,\dot y>,\{q\}>\}$ witness that $q\witforces\dot x\in\dot y$; this is not immediate, but can be seen as a result of $\<\dot z,r>\in\dot y$ and the definition of $\witforces$. Thus, $D\subseteq\{q\mid q\witforces\dot x\in\dot y\}$ being dense below $p$ implies that $p\atforces\dot x\in\dot y$.

For the forward direction, suppose that $p\atforces\dot x\in\dot y$. We want to show that the class $\{q\mid(\exists\<\dot z,r>\in\dot y)(q\leq r\,\land\,q\atforces\dot x=\dot z)\}$ is dense below $p$. Let $p_0\leq p$. By definition of $\atforces$, there is $p_1\leq p_0$ such that $p_1\witforces\dot x\in\dot y$, say witnessed by $\mathbf f$ and $\vec d$. Then our definition of $\witforces$ allows us to pick $p_2\in d^{\<p_1,\dot x,\in,\dot y>}$ and $\<\dot z,r>\in\dot y$ such that $p_2\leq r$ is compatible with $p_1$ and $\<p_2,\dot x,\subseteq,\dot z>,\<p_2,\dot z,\subseteq,\dot x>\in \mathbf f$, hence $p_2\witforces\dot x=\dot z$. Letting $q\leq p_1,p_2$, \cref{claim:at}(2) gives us $q\witforces\dot x=\dot z$. From \cref{claim:at}(2) again, it is easy to see that $q\atforces\dot x=\dot z$, so we are done.

To show that $\atforces$ satisfies \cref{def:atomic}(\ref{ad:su}), suppose first that $p\atforces\dot x\subseteq\dot y$, and let $\<\dot z,r>\in\dot x$. We want to show that $\{q\mid q\atforces\dot z\in\dot y\}$ is dense below $p\wedge r$. Suppose that $p_0\leq p,r$. By definition of $\atforces$, there is $p_1\leq p_0$ such that $p_1\witforces\dot x\subseteq\dot y$, say witnessed by $\mathbf f$ and $\vec d$. Then $e=\<p_1,\dot x,\subseteq,\dot y>\in \mathbf f$, and we can find $p_2\in d^e_{\dot z,r}$ compatible with $p_1$, for which $\<p_2,\dot z,\in,\dot y>\in\mathbf f$. Observe that $\mathbf f$ and $\vec d$ witness that $p_2\witforces\dot z\in\dot y$, so taking $q\leq p_1,p_2$, we get a $q\leq p_0$ with $q\atforces\dot z\in\dot y$ by \cref{claim:at}(2).

For the reverse direction of \cref{def:atomic}(\ref{ad:su}), and the only part in which we actually apply combinatorial pretameness, suppose that whenever $\<\dot z,r>\in\dot x$, the class ${\{q\mid q\atforces\dot z\in\dot y\}}$ is dense below $p\wedge r$. Note that the class
\[
D_{\dot z}=\{s\in P\mid(s\witforces\dot z\in\dot y)\lor(\forall t\leq s)(t\notatforces\dot z\in\dot y)\}
\]
is dense in $\PP$, and furthermore that the sequence $\<D_{\dot z}\mid\dot z\in\dom(\dot x)>$ is symmetric, since $\pi[D_{\dot z}]=D_{\pi(\dot z)}$ whenever $\pi\in\sym(\dot x)\cap\sym(\dot y)$ and $\dot z\in\dom(\dot x)$, by the symmetry lemma for $\witforces$ and $\atforces$ (\cref{claim:at}(4)). We claim that $p\atforces\dot x\subseteq\dot y$, so we need to show that $q\witforces\dot x\subseteq\dot y$ for densely many $q\leq p$. Let $p_0\le p$. By combinatorial pretameness, there is $p_1\leq p_0$ and a sequence $\<d_{\dot z}\mid\dot z\in\dom(\dot x)>\in V$ such that each $d_{\dot z}\subseteq D_{\dot z}$ is predense below $p_1$. Now, for each $\<\dot z,r>\in\dot x$, let $d_{\dot z,r}$ consist of those elements of $d_{\dot z}$ that are compatible with both $p$ and $r$.\footnote{$d_{\dot z,r}$ may not be predense below $p_1$ anymore, but it is still predense below $p_1\land r$.} Then by the definition of $D_{\dot z}$, it follows that $s\witforces\dot z\in\dot y$ whenever $s\in d_{\dot z,r}$.
Using Collection, we can obtain a set $b\in M$ such that whenever $\<\dot z,r>\in\dot x$ and $s\in d_{\dot z,r}$, there is $\<\mathbf f,\vec d>\in b$ with $\mathbf f,\vec d$ witnessing that $s\witforces\dot z\in\dot y$. By \cref{claim:at}(3), there are sets $\mathbf f,\vec d$ witnessing that $s\witforces\dot z\in\dot y$ for all $\<\dot z,r>\in\dot x$ and $s\in d_{\dot z,r}$ simultaneously. For convenience, set $w=\<p_1,\dot x,\subseteq,\dot y>$. Then, letting $\mathbf f'=\mathbf f\cup\{w\}$, $d^w=\{d_{\dot z,r}\mid\<\dot z, r>\in\dot x\}$, and $\vec d'=\vec d\cup\{\<w,d^w>\}$, it follows that $\mathbf f',\vec d'$ witness that $p_1\witforces\dot x\subseteq\dot y$. Hence, $p\atforces\dot x\subseteq\dot y$, as desired.

It is clear from how we defined $\atforces$ that \cref{def:atomic}(\ref{ad:eq}) holds, so $\SSS$ admits a forcing relation for atomic formulas.
\end{proof}

\section{Forcing Relations}\label{section:fr}

In this section, we define what it means to be a forcing relation for a given second-order formula (\cref{def:sofr}), and then show that for a first-order formula with fixed class $\SSS$-name parameters, this gives rise to a forcing relation which is itself a class (\cref{thm:fr}). However, for formulas with second-order quantifiers, these ``forcing relations'' may not exist as classes (unless we assume second-order class comprehension in addition to $\GB$) but they are still (second-order) definable. We then show that either notion of forcing relation functions as expected (Lemmas \ref{lem:sym}-\ref{lem:prov}), and compare these to the analogous forcing relations for $\PP$.

\medskip

Fix a variable $p$. The following definition is meta-theoretic.

\begin{definition}\label{def:sofr}
Suppose that $\SSS$ admits a forcing relation $\atforces$ for atomic formulas. From this we define $\sforces$, which can be viewed as a function \[\varphi(\vec x,\vec X)\quad\mapsto\quad p\sforces\varphi(\vec x,\vec X),\] which maps a second-order formula $\varphi(\vec x,\vec X)$ to a formula expressing ``$p$ forces $\varphi(\vec x,\vec X)$ over $\SSS$''. We define $\sforces$ recursively on the formulas $\varphi$ for which $p,q,r$ are not variables as follows: 
for variables $x,y,X,Y$, define
\begin{enumerate}[leftmargin=*]
\item $p\sforces x\in y$ to be $p\atforces x\in y$,\footnote{The appearance of $\atforces$ in this formula can be dealt with in one of two ways. Either we take $\atforces$ as a class parameter, or we attach a class quantifier to the beginning of the formula: $p\atforces x\in y$ becomes $(\forall X)(X\text{ is an atomic forcing relation for }\SSS\;\to\;\<p,x,\in,y>\in X)$ or $(\exists X)(X\text{ is an atomic forcing relation for }\SSS\;\land\;\<p,x,\in,y>\in X)$ (either is fine by \cref{lem:uniqueat}). Here ``$X$ is an atomic forcing relation for $\SSS$'' is a first-order formula asserting that $X$ satisfies the recursion in \cref{def:atomic}.\label{fn:atforces}}
\item $p\sforces x=y$ to be $p\atforces x=y$,
\item $p\sforces x\in X$ to be $(\forall q\leq p)(\exists r\leq q)(\exists\<\dot z,s>\in X)(r\leq s\;\land\;r\atforces x=\dot z)$,
\item $p\sforces X=Y$ to be $p\sforces(\forall x)(x\in X\ifff x\in Y)$\footnote{Here $\forall x\,\varphi$ abbreviates $\neg(\exists x)(\neg\varphi)$ and $\varphi\ifff\psi$ abbreviates $\neg(\varphi\land\neg\psi)\land\neg(\psi\land\neg\varphi)$; see (5)-(7) below. Note that this is not circular, since (5)-(7) do not depend on (4).};
\end{enumerate}
and, having defined $p\sforces\varphi(\vec x,\vec X)$ and $p\sforces\psi(\vec y,\vec Y)$, define
\begin{enumerate}[leftmargin=*]
\setcounter{enumi}{4}
\item $p\sforces(\varphi\land\psi)(\vec x\cup\vec y,\vec X\cup\vec Y)$ to be $(p\sforces\varphi(\vec x,\vec X))\land(p\sforces\psi(\vec y,\vec Y))$,
\item $p\sforces\neg\varphi(\vec x,\vec X)$ to be $(\forall q\leq p)\neg(q\sforces\varphi(\vec x,\vec X))$,
\item $p\sforces\exists z\varphi(\vec x,\vec X)$ to be $(\forall q\leq p)(\exists r\leq q)(\exists z\in\HS)(r\sforces\varphi(\vec x,\vec X))$,
\item $p\sforces\exists Z\varphi(\vec x,\vec X)$ to be $(\forall q\leq p)(\exists r\leq q)(\exists Z)(Z\text{ is an $\SSS$-name}\;\land\;r\sforces\varphi(\vec x,\vec X))$.
\end{enumerate}
\end{definition}

In the following, given a formula $\varphi$, let $\Godelnum\varphi$ be some fixed internalization of $\varphi$ that keeps track of the free variables of $\varphi$, for example via G\"odel numbering or parse trees.

\begin{definition}\label{def:fr}
A class relation $\Gforces$ is a \emph{forcing relation} (for $\SSS$) for a list of formulas $\varphi_1,\ldots,\varphi_n$ with a fixed assignment $\vec\Gamma$ of the free second-order variables of $\varphi_1,\ldots,\varphi_n$ to class $\SSS$-names\footnote{From here on, we write this as ``fixed class $\SSS$-name parameters'' to prevent clutter.} if
\[
\<p,\Godelnum{\varphi_i},\vecdot x>\in\Gforces\iffq p\sforces\varphi_i(\vecdot x,\vec\Gamma')
\]
for all $i\in\{1,\ldots,n\}$, $p\in P$, and $\vecdot x\in\HS$, where $\vec\Gamma'\subseteq\vec\Gamma$ is the assignment of the free variables of $\varphi_i$ to class $\SSS$-names. When this is the case, we write $\<p,\Godelnum{\varphi_i},\vecdot x>\in\Gforces$ as $p\Sforces\varphi_i(\vecdot x,\vec\Gamma')$.
\end{definition}

When the list of formulas above only contains first-order formulas and is closed under sub-formulas, this definition can be internalized: we could then say that a class $\Sforces$ is a forcing relation when it satisfies the semantic analogue of the (first-order part of the) syntactic recursion in \cref{def:sofr} (for some fixed class $\SSS$-name parameters) (cf. \cite[Definition 4]{ghhsw}).

\begin{theorem}\label{thm:fr}
If $\SSS$ admits a forcing relation $\atforces$ for atomic formulas, then $\SSS$ admits a forcing relation for any finite list of first-order formulas with fixed class $\SSS$-name parameters.
\end{theorem}

\begin{proof}
Looking at \cref{def:sofr} (and choosing to take $\atforces$ as a class parameter, as discussed in \cref{fn:atforces}), it is clear by meta-theoretic induction on formulas that for any first-order formula $\varphi(\vec x,\vec X)$, the formula $p\sforces\varphi(\vec x,\vec X)$ is also first-order. Thus, the result follows easily from first-order comprehension in $\GB$, since we have assigned the second-order variables to fixed class $\SSS$-names.
\end{proof}

When the conclusion (or equivalently, the assumption) of \cref{thm:fr} holds, we say that $\SSS$ \emph{admits forcing relations} and let $\Sforces$ denote the forcing relations of the system $\SSS$ -- this is of course the usual abuse of notation, since we only get class forcing relations on the condition that we have fixed some assignments of second-order variables. Even without class parameters, it may be the case that no single forcing relation works on all first-order formulas simultaneously. To reflect this, in what follows, we shall use $\vec\Gamma,\vec\Pi,\dot\Gamma,\dot\Gamma_1,\ldots$ to refer to fixed class $\SSS$-name parameters, and $\vecdot X,\vecdot Y,\dot X,\dot Y,\ldots$ to refer to general assignments of variables to class $\SSS$-names that are not fixed.

The following is immediate from \cref{lem:soatomic} and Definitions \ref{def:sofr}\&\ref{def:fr}.

\begin{lemma}\label{thm:sofr}
Assume second-order class comprehension\footnote{Given a finite list of formulas, only finitely many instances of the second-order class comprehension scheme are needed for this result.}. Then $\SSS$ admits a forcing relation for any finite list of second-order formulas with fixed class $\SSS$-name parameters.\hfill{$\qed$}
\end{lemma}

We now show that our notion of forcing relation works in the same way as we are used to. The following are actually schemes of lemmas, one for each formula $\varphi(\vec x,\vec X)$.

\begin{lemma}\label{lem:sym}
If $\SSS$ admits forcing relations, then $\sforces$ satisfies the symmetry lemma: for $\pi\in\G$, $p\in P$, and $\SSS$-names $\vecdot x,\vecdot X$,
\[
p\sforces\varphi(\vecdot x,\vecdot X)\iffq\pi(p)\sforces\varphi(\pi(\vecdot x),\pi(\vecdot X)).
\]
Consequently, $\Sforces$ also satisfies the symmetry lemma on first-order formulas with fixed class $\SSS$-name parameters.
\end{lemma}

\begin{proof}
The atomic case follows from \cref{def:sofr} and \cref{lem:atsymm}. The general case is shown by the usual induction on formula complexity in the meta-theory, using the equivalences in \cref{def:sofr}. For example, if $p\sforces(\exists Z\,\varphi)(\vecdot x,\vecdot X)$ (where $Z$ appears free in $\varphi$, the alternative being easier), then for all $q\leq p$ there is $r\leq q$ and an $\SSS$-name $\dot Z$ with $r\sforces\varphi(\vecdot x,\vecdot X\concatt\dot Z))$. But then by the induction hypothesis, given $\pi\in\G$, for all $\pi(q)\leq\pi(p)$ there is $\pi(r)\leq\pi(q)$ and an $\SSS$-name $\pi(\dot Z)$ with $\pi(r)\sforces\varphi(\pi(\vecdot x),\pi(\vecdot X)\concatt\pi(\dot Z))$. Then clearly $\pi(p)\sforces(\exists Z\,\varphi)(\pi(\vecdot x),\pi(\vecdot X))$.
\end{proof}

\newpage
\begin{lemma}\label{lem:str}
Suppose $\SSS$ admits forcing relations and $\vecdot x,\vecdot X$ are $\SSS$-names.
\begin{enumerate}
\item If $p\sforces\varphi(\vecdot x,\vecdot X)$ and $q\leq p$, then $q\sforces\varphi(\vecdot x,\vecdot X)$.
\item If whenever $q\leq p$, there is $r\leq q$ with $r\sforces\varphi(\vecdot x,\vecdot X)$, then $p\sforces\varphi(\vecdot x,\vecdot X)$.
\end{enumerate}
The same holds for $\Sforces$ when $\varphi$ is a first-order formula with fixed class $\SSS$-name parameters.
\end{lemma}

\begin{proof}
(1) is shown by induction on formula complexity. In fact, only the conjunction step requires (easy) use of the induction hypothesis; the rest follow directly from Definitions \ref{def:atomic}\&\ref{def:sofr}, since any collection of conditions dense below $p\in P$ is also dense below $q\leq p$. For example, if $q\leq p$ and $p\sforces\dot x\in\dot X$, then for any $r\leq p$, there is $s\leq r$ and $\<\dot z,t>\in\dot X$ with $s\leq t$ and $s\atforces\dot x=\dot z$; but this is certainly the case for any $r\leq q$ also, so $q\sforces\dot x\in\dot X$.

(2) is also shown by induction on formula complexity. Most steps follow directly from the definition without appealing to the inductive hypothesis; for example, if there are densely many $q\leq p$ with $q\sforces\dot x\in\dot X$, then certainly there are densely many $q\leq p$ with some $\<\dot z,r>\in\dot X$ such that $q\leq r$ and $q\atforces\dot x=\dot z$, so $p\sforces\dot x\in\dot X$. Only the conjunction and negation steps are not immediate in this way. While the former simply uses the inductive hypothesis, we provide the latter for completeness. Suppose, for contradiction, that there are densely many $q\leq p$ with $q\sforces\neg\varphi(\vecdot x,\vecdot X)$, but that $\neg(p\sforces\neg\varphi(\vecdot x,\vecdot X))$. Then, by \cref{def:sofr}, there is $r\leq p$ with $r\sforces\varphi(\vecdot x,\vecdot X)$, so by (1) we have $s\sforces\varphi(\vecdot x,\vecdot X)$ for all $s\leq r$. Our density assumption however gives us some $q\leq r$ such that $q\sforces\neg\varphi(\vecdot x,\vecdot X)$, which contradicts \cref{def:sofr}. Thus, $p\sforces\neg\varphi(\vecdot x,\vecdot X)$.
\end{proof}

The following lemma will implicitly be used all the time in the arguments that follow, as is usually the case for the analogous result for set forcing.

\begin{lemma}\label{lem:prov2}
If $\SSS$ admits forcing relations, then $\sforces$ is closed under provability.

In other words, if $\varphi\to\psi$ is a theorem of predicate logic\footnote{I.e., there is a proof of $\psi$ from $\varphi$ in a standard deduction system for predicate logic.}, and $p\sforces\varphi(\vecdot x,\vecdot X)$ for some $p\in P$ and $\SSS$-names $\vecdot x,\vecdot X$, then $p\sforces\psi(\vecdot y,\vecdot Y)$ (where $\vecdot y,\vecdot Y$ are $\SSS$-names agreeing on any shared variables\footnote{As noted before, we actually see $\vecdot x,\vecdot y,\vecdot X,\vecdot Y$ as assignments of free variables to $\SSS$-names.} with $\vecdot x,\vecdot X$).
\end{lemma}

\begin{proof}
This is shown by induction on proofs in a standard deduction system for predicate logic (with equality). We choose a Hilbert-style proof system (see \cite[Section~2.3]{mendelson}) with the inference rules:
\begin{itemize}[leftmargin=*]
\item Modus ponens: $\psi$ follows from $\varphi$ and $\varphi\to\psi$,
\item Generalization: $(\forall x\,\varphi)$ and $(\forall X\,\varphi)$ follow from $\varphi$;
\end{itemize}
and the logical axiom schemes: for any second-order formulas $\varphi,\psi,\xi$ and variables $x,y$,
\begin{enumerate}[leftmargin=*]
\item $\varphi\to(\psi\to\varphi)$,
\item $(\varphi\to(\psi\to\xi))\to((\varphi\to\psi)\to(\varphi\to\xi))$,
\item $(\neg\psi\to\neg\varphi)\to((\neg\psi\to\varphi)\to\psi)$,
\item $(\forall x\,\varphi)\to\varphi[y/x]$\footnote{We define $\varphi[y/x]$ to be the formula where every free occurrence of $x$ is replaced by $y$.}, when $y$ does not become bound in $\varphi[y/x]$,
\item $(\forall x)(\varphi\to\psi)\to(\varphi\to(\forall x\,\psi))$, when $\varphi$ contains no free occurrences of $x$,
\item $x=x$,
\item $x=y\to(\varphi\to\varphi[y/x])$, when $y$ does not become bound in $\varphi[y/x]$;
\end{enumerate}
plus the class analogues for axiom schemes (4)-(7): for any variables $X,Y$,
\begin{enumerate}[leftmargin=*]
\setcounter{enumi}{7}
\item $(\forall X\,\varphi)\to\varphi[Y/X]$, when $Y$ does not become bound in $\varphi[Y/X]$,
\item $(\forall X)(\varphi\to\psi)\to(\varphi\to(\forall X\,\psi))$, when $\varphi$ contains no free occurrences of $X$,
\item $X=X$,
\item $X=Y\to(\varphi\to\varphi[Y/X])$, when $Y$ does not become bound in $\varphi[Y/X]$.
\end{enumerate}

Since the axiom schemes above use the symbols ``$\to$'' and ``$\forall$'', whereas we have used ``$\land$'' and ``$\exists$'' in our analysis, we see $\varphi\to\psi$ as an abbreviation for $\neg(\varphi\land\neg\psi)$ and $\forall x\,\varphi$ as an abbreviation for $\neg(\exists x\,\neg\varphi)$ (and analogously for second-order quantifiers). Then, by \cref{def:sofr} and \cref{lem:str},
\begin{align*}
p\sforces(\varphi\to\psi)&\iff(\forall q\leq p)(\exists r\leq q)(r\sforces\neg\varphi\;\lor\;r\sforces\psi)\\
&\iff(\forall q\leq p)(q\sforces\varphi\;\to\;q\sforces\psi),\\
p\sforces(\forall x\,\varphi)&\iff(\forall x\in\HS)(p\sforces\varphi),\\
p\sforces(\forall X\,\varphi)&\iff(\forall X)(X\text{ is an $\SSS$-name}\to p\sforces\varphi).
\end{align*}

We show that $\sforces$ respects the inference rules. Modus ponens follows easily from our characterization of $p\sforces(\varphi\to\psi)$ -- if $p\sforces\varphi(\vecdot x,\vecdot X)$ and $p\sforces(\varphi\to\psi)(\vecdot x\cup\vecdot y,\vecdot X\cup\vecdot Y)$, then certainly $p\sforces\psi(\vecdot y,\vecdot Y)$ (where $\vecdot x,\vecdot X,\vecdot y,\vecdot Y$ are assignments of the free variables of $\varphi,\psi$ respectively to $\SSS$-names). Note that, as we have already implicitly done until now, we assert $p\sforces\varphi(\vec x,\vec X)$ (with the variables $\vec x,\vec X$ unassigned) in a particular context exactly when $p\sforces\varphi(\vecdot x,\vecdot X)$ for every assignment $\vecdot x,\vecdot X$ of $\vec x,\vec X$ to $\SSS$-names in that context. Thus, by the previous paragraph, if $p\sforces\varphi$, then both $p\sforces(\forall x\,\varphi)$ and $p\sforces(\forall X\,\varphi)$; i.e., Generalization holds.

The next step is to show that each instance of the schemes of logical axioms, with free variables assigned to $\SSS$-names, is forced by $\1$. We give some examples, and leave the rest to the reader, since they follow from \cref{def:sofr} and \cref{lem:str} in much the same way as for set forcing over $\ZF$. We suppress the $\SSS$-names for readability.

For (1), observe that
\begin{align*}
\1\sforces(\varphi\to(\psi\to\varphi))&\iffq(\forall p\in P)(\exists q\leq p)(q\sforces\neg\varphi\;\lor\;q\sforces(\psi\to\varphi))\\
&\iffq(\forall p\in P)(\exists q\leq p)(q\sforces\neg\varphi\;\lor\;q\sforces\neg\psi\;\lor\;q\sforces\varphi)).
\end{align*}
But there are certainly densely many $q\in P$ for which either $q\sforces\neg\varphi$ or $q\sforces\varphi$; this follows directly from \cref{def:sofr}. So $\1\sforces(\varphi\to(\psi\to\varphi))$. (2) and (3) are proved similarly.

For (4), suppose that $x$ appears free in $\varphi$ (the alternative is easy), and that $y$ does not become bound in $\varphi[y/x]$. Let $\vecdot x'$ be the assignment of all free set variables of $\varphi$ except $x$; we assume that this includes $y$, since that is the interesting case. Let $p\in P$, and suppose that $p\sforces(\forall x\,\varphi)(\vecdot x')$, so for all $\dot x\in\HS$, we have $p\sforces\varphi(\vecdot x'\concatt\dot x)$. Then certainly $p\sforces\varphi[y/x](\vecdot x')$, since this just amounts to assigning $x$ to the $\SSS$-name assigned to $y$. (5) and (8)-(9) are proved in a similar fashion.

We already have (6) as a result in \cref{lem:atom}(2), and (10) follows immediately from \cref{def:sofr}(4) and our characterizations of $p\sforces\forall x\,\varphi$ and $p\sforces\varphi\to\psi$ above. (7) is shown by induction on both formula complexity and name rank -- this argument is the same as the analogous one for set forcing -- and then (11) follows by induction on formula complexity. 
\end{proof}

In fact, we can restrict our attention to $\Sforces$, and still get the same result:

\begin{corollary}\label{lem:prov}
If $\SSS$ admits forcing relations, then $\Sforces$ is closed under provability.

In other words, if the first-order formula $\varphi\to\psi$ is a theorem of predicate logic and $p\Sforces\varphi(\vecdot x,\vec\Gamma)$ for some $p\in P$, $\dot x\in\HS$, and class $\SSS$-name parameters $\vec\Gamma$, then $p\Sforces\psi(\vecdot y,\vec\Pi)$ (where $\vecdot y\in\HS$ and $\vec\Pi$ are $\SSS$-name parameters which each agree on any shared variables with $\vecdot x$ and $\vec\Gamma$, respectively).\hfill{$\qed$}
\end{corollary}

This follows directly from \cref{def:fr} and \cref{lem:prov2}, and is what we will actually be using in the sections that follow, since we only get the standard truth lemma for $\Sforces$, and not $\sforces$ (see \cref{thm:truth}).

\medskip

We end this section by exploring the relationship between the relations $\Sforces$ and the usual forcing relations $\Pforces$ for $\PP$. This will be useful in situations where $\PP$ is well-behaved (for example, tame: see \cref{th:ptamestame}).

\begin{definition}\label{def:Pfr}
\begin{itemize}
\item We say that $\PP$ \emph{admits a forcing relation for atomic formulas} if there is a class relation $\atforces$ satisfying the recursion in \cref{def:atomic} for \emph{all} $\PP$-names $\dot x,\dot y\in V^{\PP}$.
\item If $\PP$ admits a forcing relation $\atforces$ for atomic formulas, we define $\pforces$ in the same way as $\sforces$ in \cref{def:sofr}, except with all references to $\HS$ and $\SSS$-names replaced by $V^{\PP}$ and $\PP$-names, respectively.
\item \cref{def:fr} is adapted to $\PP$ in the same way.
\end{itemize}
\end{definition}

If $\PP$ admits a forcing relation for any finite list of first-order formulas with fixed class $\PP$-name parameters, we say that $\PP$ \emph{admits forcing relations} and let $\Pforces$ denote the forcing relations of $\PP$.

Many of our results for $\SSS$ have analogous results for $\PP$. The version of \cref{thm:fr} for $\PP$ is given in \cite[Section 4]{hklns}: if $\PP$ admits a forcing relation for atomic formulas, then $\PP$ admits forcing relations. By \cite[Theorem 2.4]{hks}, if $\PP$ is pretame, then $\PP$ admits forcing relations. None of the proofs of Lemmas \ref{lem:sym}-\ref{lem:prov2} and \cref{lem:prov} rely on $\SSS$ being a symmetric system specifically, instead of a forcing notion, so these results hold for $\pforces$ also (when $\PP$ admits forcing relations).

If $\PP$ admits a forcing relation $\atforces$ for atomic formulas, then $\atforces$ restricted to $\HS$ clearly gives a forcing relation for atomic formulas for~$\SSS$. Thus, by \cref{thm:fr}, if $\PP$ admits forcing relations, then so does $\SSS$.

In the following, $\varphi^{\HS^\bullet}$ denotes the relativization of the formula $\varphi$ to the class $\HS^\bullet$.

\begin{lemma}\label{lem:frs}
Suppose that $\PP$ admits forcing relations $\Pforces$, and that consequently, $\SSS$ admits forcing relations $\Sforces$. Then for any given first-order formula $\varphi$ with fixed class $\SSS$-name parameters $\vec\Gamma$,
\[
p\Sforces\varphi(\vecdot x,\vec\Gamma)\quad\iff\quad p\Pforces\varphi^{\HS^\bullet}(\vecdot x,\vec\Gamma)
\]
for all $p\in P$ and $\vecdot x\in\HS$.
\end{lemma}
\begin{proof}
This is shown by induction on formula complexity. We already discussed the atomic case. Clearly, the result also holds for any formula of the form $x\in\dot\Gamma$ with $\SSS$-name $\dot\Gamma$ (in particular, with the $\SSS$-name $\HS^\bullet$).

For the inductive steps, suppose the result holds for the formula $\varphi$ (with class $\SSS$-name parameters $\vec\Gamma$).
The Boolean steps are straightforward; we provide negation as an example. We have $p\Sforces\neg\varphi(\dot x,\vec\Gamma)$ if and only if there is no $q\leq p$ with $q\Sforces\varphi(\dot x,\vec\Gamma)$, if and only there is no $q\leq p$ with $q\Pforces\varphi^{\HS^\bullet}(\dot x,\vec\Gamma)$, if and only if $p\Pforces\neg\varphi^{\HS^\bullet}(\dot x,\vec\Gamma)$ (which is the same as $p\Pforces(\neg\varphi)^{\HS^\bullet}(\dot x,\vec\Gamma)$).

For the quantification step, suppose first that $p\Sforces(\exists z\,\varphi)(\vecdot x,\vec\Gamma)$ (where $z$ appears free in $\varphi$; the alternative is easy). Then
\[
\{q\in P\mid(\exists\dot z\in\HS)(q\Sforces\varphi(\vecdot x\concatt\dot z,\vec\Gamma))\}
\]
is dense below $p$, so certainly $\{q\in P\mid(\exists\dot z\in V^{\PP})(q\Sforces(\dot z\in\HS^\bullet\,\land\,\varphi(\vecdot x\concatt\dot z,\vec\Gamma)))\}$ is dense below $p$. Then the induction hypothesis (and the result for the formula $x\in\HS^\bullet$) imply that $\{q\in P\mid(\exists\dot z\in V^{\PP})(q\Pforces(\dot z\in\HS^\bullet\,\land\,\varphi^{\HS^\bullet}(\vecdot x\concatt\dot z,\vec\Gamma)))\}$ is dense below $p$. So $p\Pforces(\exists z)(z\in \HS^\bullet\,\land\,\varphi^{\HS^\bullet}(\vecdot x\concatt z,\vec\Gamma))$ (i.e.\ $p\Pforces(\exists z\,\varphi)^{\HS^\bullet}(\vecdot x,\vec\Gamma)$). For the reverse direction, suppose that $p\Pforces(\exists z\,\varphi)^{\HS^\bullet}(\vecdot x,\vec\Gamma)$). Then 
\[
\{q\in P\mid(\exists\dot z\in V^{\PP})(q\Pforces(\dot z\in\HS^\bullet\,\land\,\varphi^{\HS^\bullet}(\vecdot x\concatt\dot z,\vec\Gamma)))\}
\]
is dense below $p$, so by \cref{def:Pfr},
\[
\{q\in P\mid(\exists\dot z\in V^{\PP})(\exists\dot u\in\HS)(q\Pforces(\dot z=\dot u\,\land\,\varphi^{\HS^\bullet}(\vecdot x\concatt\dot z,\vec\Gamma)))\}
\]
is dense below $p$. A straightforward induction on the complexity of $\varphi$ then gives us that $\{q\in P\mid(\exists\dot z\in\HS)(q\Pforces\varphi^{\HS^\bullet}(\vecdot x\concatt\dot z,\vec\Gamma)\}$ is dense below $p$. From the induction hypothesis and \cref{def:sofr}(7), we conclude that $p\Sforces(\exists z\,\varphi)(\vecdot x,\vec\Gamma)$.
\end{proof}

\section{The truth lemma}\label{section:truth}

We want to show that the statement commonly known as the truth lemma for forcing holds in the context of class symmetric systems. For a meaningful version of this theorem, we work in an ambient universe $V$ of set theory in which there is a transitive model $\MM=\<M,\in,\C>\models\GB$,\footnote{That is, both $M$ and $\C$ are transitive.} with a domain $M$ of sets and a domain $\C$ of classes, whose elementhood relation $\in$ coincides with (the restriction to $\C$ of) that of $V$.\footnote{Even if $\MM$ does not satisfy these two conditions, as long as its elementhood relation $\in_M$ is set-like on $M$ and well-founded, and $V$ satisfies enough set theory to perform a Mostowski collapse on~$\MM$, then we can get a model isomorphic to $\MM$ which is transitive and whose elementhood relation coincides with that of $V$. If $\in_M$ is set-like on $M$ but not well-founded, we can still construct an extension $\MMGS$ whose objects are equivalence classes of $\SSS$-names forced to be equal by an element of $G$; for this construction, the atomic case of the truth lemma is immediate.} We assume that $\SSS=\<\PP,\G,\F>$ denotes a class symmetric system in $\MM$; the interesting case is when \emph{$\SSS$-generic filters} exist in $V$.

\begin{definition}
\begin{itemize}
\item A filter $G\subseteq P$ is \emph{$\SSS$-generic} in case it intersects every symmetrically dense subclass of $\PP$ in $\MM$.
\item Given a $\PP$-name $\dot x$, we may recursively evaluate it by an $\SSS$-generic filter as usual, letting $\dot x^G=\{\dot y^G\mid(\exists p\in G)(\<\dot y,p>\in\dot x)\}$.
\item An $\SSS$-generic extension by an $\SSS$-generic filter $G$ is of the form $\MMGS=\<M[G]_\SSS,\in,\C[G]_\SSS>$, where
\[
M[G]_\SSS=\{\dot x^G\mid\dot x\in\HS\}\textrm{, and}
\]
\[
\C[G]_\SSS=\{\dot X^G\mid\dot X\in\C\textrm{ is an $\SSS$-name}\}.
\]
We also refer to such an extension as a symmetric extension (of $\MM)$.
\end{itemize}
\end{definition}

\begin{theorem}\label{thm:truth}
If $\SSS$ admits forcing relations, then for any given first-order formula~$\varphi$ with fixed class $\SSS$-name parameters $\vec\Gamma$, whenever $G\subseteq P$ is $\SSS$-generic and $\vecdot x\in\HS$,
\begin{equation}\tag{$\ast$}\label{ast}
\MMGS\models\varphi(\vec{\dot{x}}^G,\vec{\Gamma}^G)\iffq(\exists p\in G)(p\Sforces\varphi(\vecdot x,\vec\Gamma)).
\end{equation}
\end{theorem}

\begin{proof}
The proof is similar to the analogous result for class forcing, except we must pay attention to the symmetry of the relevant dense classes.
Suppose that $G\subseteq P$ is $\SSS$-generic. We show (\ref{ast}) by induction on the complexity of the formula $\varphi$.

It is useful to establish that (\ref{ast}) plays well with negation and conjunction before we prove the atomic case. So, suppose first that (\ref{ast}) holds for $\varphi$ and let $\vecdot x\in\HS$. If $p\Sforces\neg\varphi(\vecdot x,\vec\Gamma)$ for some $p\in G$, then no element of $G$ forces $\varphi(\vecdot x,\vec\Gamma)$ (by \cref{lem:str}(1)), so our assumption on $\varphi$ tells us that $\neg\varphi(\vecdot x^G,\vec\Gamma^G)$ holds in $\MMGS$. Conversely, if $\neg\varphi(\vecdot x^G,\vec\Gamma^G)$ holds in $\MMGS$, then by our assumption on $\varphi$, there is no $p\in G$ with $p\Sforces\varphi(\vecdot x,\vec\Gamma)$. But the class
\[
A=\{q\in P\mid(q\Sforces\neg\varphi(\vecdot x,\vec\Gamma))\lor(q\Sforces\varphi(\vecdot x,\vec\Gamma))\}
\]
is dense in $\PP$ by \cref{def:sofr}(6), and symmetric by \cref{lem:sym}. We may thus pick $q\in G\cap A$. Since $p,q\in G$ are compatible, by \cref{lem:str}(1) it cannot be the case that $q\Sforces\varphi(\vecdot x,\vec\Gamma)$, so it must be the case that $q\Sforces\neg\varphi(\vecdot x,\vec\Gamma)$.

For conjunction, suppose that (\ref{ast}) holds for the formulas $\varphi,\psi$ with class parameters $\vec\Gamma,\vec\Gamma'$ agreeing on any shared variables, and let $\vecdot x,\vecdot y\in\HS$ be assignments of the free set variables of $\varphi,\psi$ respectively, that also agree on any shared variables. Then $p\Sforces(\varphi\land\psi)(\vecdot x\cup\vecdot y,\vec\Gamma\cup\vec\Gamma')$ for some $p\in G$ if and only if $p\Sforces\varphi(\vecdot x,\vec\Gamma)$ and $p\Sforces\psi(\vecdot y,\vec\Gamma')$ for some $p\in G$ (by \cref{lem:str}(1)), if and only if $\varphi(\vecdot x^G,\vec\Gamma^G)$ and $\psi(\vecdot y^G,\vec\Gamma'^G)$ hold in $\MMGS$, if and only if $(\varphi\land\psi)(\vecdot x^G\cup\vecdot y^G,\vec\Gamma^G\cup\vec\Gamma'^G)$ holds in $\MMGS$.

\medskip

With the Boolean operations under our belt, we now show by induction on name ranks that (\ref{ast}) holds for the atomic formulas $x\in y$, $x\subseteq y$, $x=y$. The base case (all name ranks $0$) is immediate from the definitions.

Suppose first that $p\Sforces\dot x\in\dot y$ for some $p\in G$ and $\dot x,\dot y\in\HS$. The class
\[
B=\{q\in P\mid(\exists\<\dot z,r>\in\dot y)(q\leq r\,\land\,q\Sforces\dot x=\dot z)\lor(q\Sforces\dot x\notin\dot y)\}
\]
is dense in $\PP$: if $s\notSforces\dot x\notin\dot y$ for all $s\leq t$, then $t\Sforces\dot x\in\dot y$ by \cref{def:sofr}(6), and so by \cref{def:atomic}(\ref{ad:el}) there are densely many $q\leq t$ with $\<\dot z,r>\in\dot y$ such that $q\leq r$ and $q\Sforces\dot x=\dot y$. $B$ is symmetric, with $\sym(B)\supseteq\sym(\dot x)\cap\sym(\dot y)\in\F$ by \cref{lem:sym}. Thus, we can take $q\in G\cap B$, for which there must be a $\<\dot z,r>\in\dot y$ with $q\leq r$ and $q\Sforces\dot x=\dot z$ by \cref{lem:str}(1). By our induction hypothesis, since $\rho(\dot z)<\rho(\dot y)$, $\dot x^G=\dot z^G$, and $\dot z^G\in\dot y^G$ by \cref{lem:atom}(1), so $\dot x^G\in\dot y^G$.

Conversely, suppose that $\dot x^G\in\dot y^G$ for some $\dot x,\dot y\in\HS$. Then there is $\<\dot z, r>\in\dot y$ with $\dot x^G=\dot z^G$ and $r\in G$. By \cref{lem:atom}(1), $r\Sforces\dot z\in\dot y$. By our inductive hypothesis, there is $p\in G$ with $p\Sforces\dot x=\dot z$; by \cref{lem:str}(1) we can assume $p\leq r$ since both are in $G$. Therefore, $p\Sforces\dot x\in\dot y$, as desired.

Now, suppose $p\Sforces\dot x\subseteq\dot y$ for some $p\in G$. To show that $\dot x^G\subseteq\dot y^G$, consider $a\in\dot x^G$; there is $\<\dot z,r>\in\dot x$ such that $a=\dot z^G$ and $r\in G$. By \cref{def:sofr}(6), the class
\[
C=\{q\in P\mid(q\Sforces\dot z\in\dot y)\lor(q\Sforces\dot z\not\in\dot y)\}
\]
is dense in $\PP$, and symmetric by \cref{lem:sym}. Thus, we can pick a $q\in G\cap C$, for which it must be that $q\Sforces\dot z\in\dot y$ by \cref{def:atomic}(\ref{ad:su}). Then $\dot z^G\in\dot y^G$ by our inductive hypothesis.

Conversely, suppose that $\dot x^G\subseteq\dot y^G$ for some $\dot x,\dot y\in\HS$; we want to find $p\in G$ with $p\Sforces\dot x\subseteq\dot y$. Consider the class
\[
D=\{p\in P\mid(p\Sforces\dot x\subseteq\dot y)\lor(\exists\<\dot z,r>\in\dot x)(p\leq r\,\land\,p\Sforces\dot z\notin\dot y)\}.
\]
We argue that $D$ is dense in $\PP$: if $s\notSforces\dot x\subseteq\dot y$ for all $s\leq t$, then by Definitions \ref{def:atomic}(\ref{ad:su})\&\ref{def:sofr}(6) there are densely many $q\leq t$ with a $\<\dot z,r>\in\dot x$ such that $q\leq r$ and $q\Sforces\dot z\notin\dot y$. Since $D$ is symmetric, we can take $p\in G\cap D$. Suppose, for contradiction, that there is $\<\dot z,r>\in\dot x$ with $p\leq r$ and $p\Sforces\dot z\notin\dot y$, so $r\in G$ and $\dot z^G\in\dot x^G$. Then $\dot z^G\in\dot x^G\setminus\dot y^G$ by our inductive hypothesis and our argument that (\ref{ast}) plays well with negation. But then $\dot x^G\not\subseteq\dot y^G$, a contradiction. So there is $p\in G$ with $p\Sforces\dot x\subseteq\dot y$.

Finally, the fact that (\ref{ast}) holds for set equality is due to the transitivity of $M$, from which we get Extensionality for sets in $\MMGS$.

\medskip

Moving on to class elementhood, suppose that $p\Sforces\dot x\in\dot\Gamma$ for some $\dot x\in\HS$, $\SSS$-name $\dot\Gamma$, and $p\in G$. By \cref{def:sofr}(3), the class
\[
E=\{q\in P\mid(\exists\<\dot z,r>\in\dot\Gamma\ q\leq r\,\land\,q\Sforces\dot x=\dot z)\lor(q\forces\dot x\not\in\dot\Gamma)\}
\]
is dense in $\PP$. Since $E$ is also symmetric, we can take $q\in G\cap E$, for which there must be $\<\dot z,r>\in\dot\Gamma$ with $q\leq r$ and $q\Sforces\dot x=\dot z$. By (\ref{ast}) for set equality, we then get $\dot x^G=\dot z^G\in\dot\Gamma^G$. Conversely, if $\dot x^G\in\dot\Gamma^G$, then there is some $\<\dot z,r>\in\dot\Gamma$ with $\dot x^G=\dot z^G$ and $r\in G$. By \cref{lem:atom}(1), $r\Sforces\dot z\in\dot\Gamma$. Now, (\ref{ast}) for set equality yields $p\in G$ with $p\Sforces\dot x=\dot z$, and we may also assume $p\le r$. It follows that $p\Sforces\dot x\in\dot\Gamma$.

For the quantification step of the induction, suppose that (\ref{ast}) holds for the formula $\varphi$ (with class parameters $\vec\Gamma$, and $z$ appearing free), and let $\vecdot x\in\HS$. For the forward direction: if $\MMGS$ satisfies $(\exists z\,\varphi)(\vecdot x^G,\vec\Gamma^G)$, then there is some $\dot z\in\HS$ for which $\varphi(\vecdot x^G\concatt\dot z^G,\vec\Gamma^G)$ holds in $\MMGS$, and for this $\dot z$ there is $p\in G$ with $p\Sforces\varphi(\vecdot x\concatt\dot z,\vec\Gamma)$ by the inductive hypothesis. But then $p\Sforces(\exists z\,\varphi)(\vecdot x,\vec\Gamma)$ by \cref{def:sofr}(7) and \cref{lem:str}(1). Conversely, suppose that $p\Sforces(\exists z\,\varphi)(\vecdot x,\vec\Gamma)$ for some $p\in G$, and consider the class
\[
F=\{q\in P\mid(\exists\dot z\in\HS)(q\Sforces\varphi(\vecdot x\concatt\dot z,\vec\Gamma))\lor(q\Sforces\neg(\exists z\,\varphi(\vecdot x,\vec\Gamma)))\},
\]
which is symmetrically dense in $\PP$ by \cref{def:sofr}(6)\&(7), so we can take $q\in G\cap F$. By our assumption on $p\in G$ and \cref{lem:str}(1), there must be some $\dot z\in\HS$ such that $q\Sforces\varphi(\vecdot x\concatt\dot z,\vec\Gamma)$. Then our inductive assumption tells us that $\MMGS$ satisfies $\varphi(\vecdot x^G\concatt\dot z^G,\vec\Gamma^G)$, so $(\exists z\,\varphi)(\vecdot x^G,\vec\Gamma^G)$ holds in $\MMGS$.

Finally, (\ref{ast}) holds for formulas of the form $X=Y$ because this is defined from formulas of the form $x\in X$ (which satisfy (\ref{ast})) using Boolean operations and quantifiers (which play well with (\ref{ast})), and because $\MMGS$ automatically satisfies Extensionality for classes, by the transitivity of $M$.
\end{proof}

\begin{corollary}\label{cor:truth}
Suppose that $\SSS$ admits forcing relations, and that for any $p\in P$, there is an $\SSS$-generic $G\subseteq P$ such that $p\in G$. Then for any given first-order formula~$\varphi$ with fixed class $\SSS$-name parameters $\vec\Gamma$, and any $p\in P$ and $\vecdot x\in\HS$,
\[
p\Sforces\varphi(\vecdot x,\vec\Gamma)\iffq(\forall\text{ $\SSS$-generic }G\subseteq P)(p\in G\to\MMGS\models\varphi(\vecdot x^G,\vec\Gamma^G)).
\]
\end{corollary}

\begin{proof}
The proof works the same as for usual forcing. The forward direction follows directly from \cref{thm:truth}. We show the contrapositive of the reverse direction: suppose that $p\notSforces\varphi(\vecdot x,\vec\Gamma)$. Then there is some $q\leq p$ with $q\Sforces\neg\varphi(\vecdot x,\vec\Gamma)$. Let $G\subseteq P$ be an $\SSS$-generic with $q\in G$ (so $p\in G$ also). Then by \cref{thm:truth}, $\MMGS\models\neg\varphi(\vecdot x^G,\vec\Gamma^G)$, so we are done.
\end{proof}

\begin{corollary}\label{cor:sotruth}
Assume second-order class comprehension. Then, for any given second-order formula $\varphi$, whenever $G\subseteq P$ is $\SSS$-generic and $\vecdot x,\vecdot X$ are $\SSS$-names,
\[
\MMGS\models\varphi(\vecdot x^G,\vecdot X^G)\iffq(\exists p\in G)(p\sforces\varphi(\vecdot x,\vecdot X).
\]
Furthermore, if for any $p\in P$, there is an $\SSS$-generic $G\subseteq P$ such that $p\in G$, then
\[
p\sforces\varphi(\vecdot x,\vecdot X)\iffq(\forall\text{ $\SSS$-generic }G\subseteq P)(p\in G\to\MMGS\models\varphi(\vecdot x^G,\vecdot X^G)).
\]
\end{corollary}

\begin{proof}
Most of the proof is exactly the same as those of \cref{thm:truth} and \cref{cor:truth}, except using $\sforces$ instead of $\Sforces$, noting that second-order class comprehension allows us to form the classes of conditions required for $G$ to intersect. The class quantification step of the induction works in the same way as the set quantification step of the former proof.
\end{proof}

If we did not have second-order class comprehension, but we still wanted a version of the truth lemma for second-order formulas, we would need a stronger notion of genericity which ensures that the filter $G$ meets all symmetrically dense conditions that are \emph{second-order definable} over $\MM$ -- for example, a similar approach is taken in \cite{chuaqui} for the case of usual class forcing.

\section{Preserving \texorpdfstring{$\GB^-$}{GB⁻}}\label{section:axioms}

We go back to working in $\GB$; we again assume that $\SSS=\<\PP,\G,\F>$ is a class symmetric extension. We show that all axioms of $\GB^{-}$ except Separation and Collection are in fact automatically preserved by $\SSS$ as long as $\SSS$ admits forcing relations (\cref{lem:pres}). We then introduce the notions of stratification (\cref{def:strat}), and show that combinatorial pretameness and stratification of $\SSS$ is enough to preserve every axiom of $\GB^{-}$ (\cref{thm:axioms}). However, we note that neither condition is necessary for this preservation.

\begin{lemma}\label{lem:pres}
If $\SSS$ admits forcing relations, then the following hold.
\begin{enumerate}
\item If $\varphi$ is Extensionality, Regularity, Pairing, Union, or Infinity, then $\1\Sforces\varphi$.
\item For any class $\SSS$-names $\dot\Gamma$ and $\dot\Pi$, $\1\Sforces((\forall x)(x\in\dot\Gamma\ifff x\in\dot\Pi)\to\dot\Gamma=\dot\Pi)$.
\item For any $\dot x\in\HS$, there is a class $\SSS$-name $\dot\Gamma$ with $\1\Sforces(\forall y)(y\in\dot x\ifff y\in\dot\Gamma)$.
\item If $\forall X_1\ldots\forall X_n\exists Y\varphi(X_1,\ldots,X_n,Y)$ is a Class Existence axiom, where $\varphi$ is first-order\footnote{All Class Existence axioms can be written in this form.}, then for all class $\SSS$-names $\dot\Gamma_1,\ldots,\dot\Gamma_n$, there is a class $\SSS$-name $\dot\Pi$ such that $\1\Sforces(\forall x)(x\in\dot\Pi\ifff\varphi(x,\dot\Gamma_1,\ldots,\dot\Gamma_n))$.
\end{enumerate}
\end{lemma}

\begin{proof}
For Extensionality, given $\dot x,\dot y\in \HS$, it is straightforward to show that $\1\Sforces((\forall z)(z\in\dot x\ifff z\in\dot y)\to\dot x=\dot y)$ from Definitions \ref{def:atomic}\&\ref{def:sofr} and the results in \cref{section:fr}. Now suppose, for contradiction, that Regularity is not forced by $\1$, so there is $\dot x\in\HS$ and $p\in G$ with
\[
p\Sforces(\dot x\neq\emptyset\;\land\;(\forall y\in\dot x)(y\cap\dot x\neq\emptyset)).
\]
By Regularity, there is $\dot y\in\HS$ of least name rank that has some $q\leq p$ with $q\Sforces\dot y\in\dot x$. By assumption on $p$, the class
\[
D_{\dot y}=\{s\in P\mid(\exists\<\dot z,r>\in\dot y)(s\leq r\;\land\;s\Sforces\dot z\in\dot x)\}
\]
is dense below $q$. So there exists $s\leq q$ and $\<\dot z,r>\in\dot y$ with $s\leq r$ and $s\Sforces\dot z\in\dot x$. But this $\dot z$ has lower name rank than $\dot y$, a contradiction. For Pairing, given $\dot x,\dot y\in\HS$, the $\PP$-name $\{\dot x,\dot y\}^\bullet$ is hereditarily symmetric (since $\sym(\dot x)\cap\sym(\dot y)\in\F)$ and $\1\Sforces(\forall z)(z\in\{\dot x,\dot y\}^\bullet\;\ifff\;z=\dot x\lor z=\dot y)$. To show (the weak form of) Union, suppose $\dot x\in\HS$ and let $u=\bigcup\{\dom(\dot y)\mid\dot y\in\dom(\dot x)\}$. Then $\pi(u^\bullet)=u^\bullet$ for any $\pi\in\sym(\dot x)$, and $\1\Sforces(\forall y\in u^\bullet)(y\subseteq u^\bullet)$. Checking Infinity is easy, since $\check\omega$ is hereditarily symmetric and $\1\Sforces(\check\omega\text{ is an inductive set})$.

(2) is immediate from \cref{def:sofr}(4), and for (3), any $\dot x\in\HS$ has a class $\SSS$-name $\dot\Gamma$ for which $(\forall y)(y\in\dot x\ifff y\in\dot\Gamma)$, and then it is not hard to see that $\1\Sforces(\forall y)(y\in\dot x\ifff y\in\dot\Gamma)$.

We show (4) for first-order class comprehension instead of the Class Existence axioms, since the former clearly implies the latter. Given a first-order formula $\varphi$ with $\SSS$-name parameters $\vec\Gamma$, the class
\[
\dot\Pi=\{\<\dot x,p>\mid\dot x\in\HS\,\land\,p\in P\,\land\,p\Sforces\varphi(\dot x,\vec\Gamma)\}
\]
is an $\SSS$-name (by symmetry of $\vec\Gamma$ and \cref{lem:sym}), and $\1\Sforces(\forall x)(x\in\dot\Pi\;\ifff\;\varphi(x,\vec\Gamma))$.
\end{proof}

Note that (2) is equivalent to $\1\sforces\text{``Extensionality for classes''}$, while (3) implies $\1\sforces$``every set represents a class'' and (4) implies $\1\sforces\varphi$ for any Class Existence axiom $\varphi$. We get a similar result for second-order comoprehension:

\begin{lemma}\label{lem:sopres}
Assume second-order class comprehension. Then, given a second-order formula $\varphi$ with class $\SSS$-name parameters $\vec\Gamma$, the class
\[
\dot\Pi=\{\<\dot x,p>\mid\dot x\in\HS\,\land\,p\in P\,\land\,p\sforces\varphi(\dot x,\vec\Gamma)\}
\]
is an $\SSS$-name with $\1\sforces(\forall x)(x\in\dot\Pi\;\ifff\;\varphi(x,\vec\Gamma))$.\hfill{$\qed$}
\end{lemma}

The above conclusion implies $\1\sforces\psi$ for the instance $\psi$ of the second-order class comprehension scheme involving $\varphi$.

\medskip

Having shown that, given $\SSS$ admits forcing relations, $\SSS$ preserves all axioms of $\GB^-$ except for possibly Separation and Collection, we define pretameness of $\SSS$ by the preservation of these two axioms.

\begin{definition}\label{def:pretame}
A class symmetric system $\SSS$ is \emph{pretame} if it admits forcing relations and preserves Separation and Collection, i.e., for any $\SSS$-name~$\dot\Gamma$,
\begin{itemize}[leftmargin=*]
\item $\1\Sforces(\forall x\exists y)(y=x\cap\dot\Gamma)$, and
\item $\1\Sforces(\forall x)\big[(\dot\Gamma\text{ is a relation}\land(\forall u\in x)(\exists v)(\<u,v>\in\dot\Gamma))\to$\\
\strut\hfill$(\exists y)(\forall u\in x)(\exists v\in y)(\<u,v>\in\dot\Gamma)\big]$.
\end{itemize}
\end{definition}

Observe that, by \cref{lem:pres}, $\SSS$ is pretame if and only if $\SSS$ admits forcing relations and $\1\sforces\varphi$ whenever $\varphi$ is an axiom of $\GB^-$.

\medskip

These results translate smoothly into the context of \cref{section:truth}, where $\MM$ is a transitive model of $\GB$ in an ambient universe $V$ of set theory:

\begin{lemma}
\begin{enumerate}
\item If $\SSS$ admits forcing relations, then $\MMGS\models\varphi$ for any axiom $\varphi$ of $\GB^-$ besides Separation or Collection, and any $\SSS$-generic $G\subseteq P$.
\item If $\SSS$ is pretame, then $\MMGS\models\varphi$ for any axiom $\varphi$ of $\GB^-$ and any $\SSS$-generic $G\subseteq P$.
\item Suppose that $\SSS$ admits forcing relations and for any $p\in P$, there is an $\SSS$-generic containing $p$. Then the implication in (2) can be reversed.
\end{enumerate}
\end{lemma}

\begin{proof}
(1) and (2) are direct consequences of \cref{lem:pres}, \cref{def:pretame}, and \cref{thm:truth}.

To see (3), suppose that $\MMGS\models\varphi$ for any $\SSS$-generic $G\subseteq P$ and axiom $\varphi$ of $\GB^-$, and that for any $p\in P$, there is an $\SSS$-generic containing $p$. Then, by \cref{cor:truth}, for any $\SSS$-name $\dot\Gamma$, both bullet points of \cref{def:pretame} hold.
\end{proof}

From \cref{lem:sopres} and \cref{cor:sotruth}, we also get the following:

\begin{corollary}
Suppose $\MM$ satisfies second-order class comprehension. Then for any $\SSS$-generic $G\subseteq P$, the model $\MMGS$ also satisfies second-order class comprehension.\hfill{$\qed$}
\end{corollary}

Moving out of the context of \cref{section:truth} and back to $\GB$, we give an additional condition on class symmetric systems, which, along with combinatorial pretameness, ensures the preservation of Separation and Collection.

\begin{definition}\label{def:strat}
A class symmetric system $\SSS$ is \emph{stratified} if for every set $a\subseteq P$, there is a symmetric set $b\subseteq P$ with $a\subseteq b$.
\end{definition}

If global choice holds, then stratification of $\SSS$ implies the existence of a $\subseteq$-increasing sequence $\<P_\alpha\mid\alpha\in\Ord>$ of symmetric subsets of~$P$ with $P=\bigcup_{\alpha\in\Ord}P_\alpha$ (the reverse always holds, as witnessed by the $P_\alpha$). This is why we called this property stratification.

For any symmetric set $b\subseteq P$, let $\HS_\alpha^{b}$ denote the set of names in $\HS_\alpha$ which have only elements of $b$ appearing\footnote{We say that $q$ appears in a name $\dot a$ if there is $\<r,\dot b>\in\dot a$ such that $q=r$ or $q$ appears in $\dot b$, inductively.} in them. Since permutations of $\PP$ preserve name rank, $(\HS_\alpha^b)^\bullet$ is then itself in $\HS$, with $\sym((\HS_\alpha^b)^\bullet)\supseteq\sym(b)\in\F$. 

\medskip

By similar arguments as in \cite[Theorem 3.1]{hks}, we can now show the following:

\begin{theorem}\label{thm:axioms}
If $\SSS$ is combinatorially pretame and stratified, then $\SSS$ is pretame.
\end{theorem}

\begin{proof}
Suppose that $\SSS$ is combinatorially pretame and stratified. By the results of \cref{section:forcingtheorem}, we already know that $\SSS$ admits forcing relations. Thus, we just need to show that $\SSS$ preserves Separation and Collection.

To show Separation, suppose that $\dot z\in\HS$ and $\dot\Gamma$ is a class $\SSS$-name. For each $\dot x\in\dom(\dot z)$, let
\[
D_{\dot x}=\{q\in P\mid(q\Sforces\dot x\in\dot\Gamma)\lor(q\Sforces\dot x\notin\dot\Gamma)\}.
\]
Each $D_{\dot x}$ is dense in $\PP$, and the sequence $\<D_{\dot x}\mid\dot x\in\dom(\dot z)>$ is symmetric because $\pi[D_{\dot x}]=D_{\pi(\dot x)}$ whenever $\pi\in\sym(\dot z)\cap\sym(\dot\Gamma)$ and $\dot x\in\dom(\dot z)$ (by \cref{lem:sym}). By combinatorial pretameness of $\SSS$, the class
\[
A=\{p\in P\mid(\exists\<d_{\dot x}\mid\dot x\in\dom(\dot z)>\in V)(\forall\dot x\in\dom(\dot z))(d_{\dot x}\subseteq D_{\dot x}\text{ is predense }{\leq}p)\}
\]
is dense in $\PP$. We show that $p\Sforces(\exists u)(u=\dot z\cap\dot\Gamma)$ whenever $p\in A$; since $A$ is dense in $\PP$, this would imply that $\1\Sforces(\exists u)(u=\dot z\cap\dot\Gamma)$ by \cref{lem:str}(2). To this end, let $p\in A$ be fixed and let $\<d_{\dot x}\mid\dot x\in\dom(\dot z)>$ be a set such that $d_{\dot x}\subseteq D_{\dot x}$ and $d_{\dot x}$ is predense below $p$ whenever $\dot x\in\dom(\dot z)$. By stratification of $\SSS$, there is a symmetric set $b\subseteq P$ such that each $d_{\dot x}\subseteq b$. Then the set
\[
\dot u=\{\<\dot y,s>\mid s\in b\,\land\,(\exists r)(\<\dot y,r>\in\dot z\,\land\,s\leq r\,\land\,s\Sforces\dot y\in\dot\Gamma)\}
\]
is hereditarily symmetric and $p\Sforces\dot u=\dot z\cap\dot\Gamma$, so certainly $p\Sforces(\exists u)(u=\dot z\cap\dot\Gamma)$.

We move on to Collection. Suppose that $\dot\Gamma$ is an $\SSS$-name, $\dot z\in\HS$, and $p\in P$ such that
\[
p\Sforces(\dot\Gamma\text{ is a relation}\;\land\;(\forall x\in\dot z)(\exists y)(\<x,y>\in \dot\Gamma)).
\]
For each $\dot x\in\dom(\dot z)$, consider the class
\[
D_{\dot x}=\{s\in P\mid (\exists\dot y\in\HS)(s\Sforces\<\dot x,\dot y>\in\dot\Gamma)\lor(\forall\dot y\in\HS)(s\Sforces\<\dot x,\dot y>\notin\dot\Gamma)\}.
\]
Observe that $D_{\dot x}$ is dense in $\PP$ by \cref{def:sofr}(6): if there is $r\in P$ with no $s\leq r$ or $\dot y\in\HS$ such that $s\Sforces\<\dot x,\dot y>\in\dot\Gamma$, then $r\Sforces\<\dot x,\dot y>\notin\dot\Gamma$ for all $\dot y\in\HS$. Also, the sequence $\<D_{\dot x}\mid\dot x\in\dom(\dot z)>$ is symmetric by symmetry of $\dot z$ and $\dot\Gamma$ and \cref{lem:sym}. Note that any element $s$ of $D_{\dot x}$ below $p$ must have a $\dot y\in\HS$ with $s\Sforces\<\dot x,\dot y>\in\dot\Gamma$, otherwise we contradict that $p\Sforces(\forall x)(\exists y)(\<x,y>\in \dot\Gamma)$. As in the argument for Separation, by combinatorial pretameness of $\SSS$, the class
\[
B=\{q\in P\mid(\exists\<d_{\dot x}\mid\dot x\in\dom(\dot z)>\in V)(\forall\dot x\in\dom(\dot z))(d_{\dot x}\subseteq D_{\dot x}\text{ is predense }{\leq}q)\}
\]
is dense in $\PP$. We show that $q\Sforces(\exists u)(\forall x\in\dot z)(\exists y\in u)(\<x,y>\in\dot\Gamma)$ whenever $q\leq p$ and $q\in B$, whereby $p\Sforces(\exists u)(\forall x\in\dot z)(\exists y\in u)(\<x,y>\in\dot\Gamma)$ by \cref{lem:str}(2). So, fix $q\leq p$ in $B$ and let $\<d_{\dot x}\mid\dot x\in\dom(\dot z)>$ be a set such that each $d_{\dot x}\subseteq D_{\dot x}$ is predense below $q$. We can now use stratification of $\SSS$ and Collection to find a symmetric set $b\subseteq P$ and an ordinal $\alpha$ such that there is $\dot y\in\HS_\alpha^b$ with $s\Sforces\<\dot x,\dot y>\in\dot\Gamma$ whenever $\<\dot x,r>\in\dot z$ and $s\in d_{\dot x}$ is compatible with $r$ and $q$. But this means that
\[
q\Sforces(\forall x\in\dot z)(\exists y\in(\HS_\alpha^b)^\bullet)(\<x,y>\in\dot\Gamma),
\]
as desired (since $(\HS_\alpha^b)^\bullet\in\HS$).
\end{proof}

The above should be contrasted with a result of Matthews (\cite[Section 4.1]{matthews} or \cite[Theorem 7.13]{MATTHEWSpaper}), where he shows that one can have a pretame notion of class forcing $\PP$ with a symmetric system $\SSS=\<\PP,\G,\F>$ -- which is thus combinatorially pretame -- that fails to force Replacement, which tells us in particular that only assuming combinatorial pretameness of $\SSS$ is not sufficient for the pretameness of~$\SSS$. It also answers the second question posed in \cite[Question 7.14]{MATTHEWSpaper}, by providing stratification as a very natural combinatorial property of a symmetric system ${\SSS=\<\PP,\G,\F>}$ that ensures, together with the pretameness of $\PP$, the preservation of $\GB^-$ and thus the pretameness of $\SSS$.\footnote{Stratification of $\SSS$ however is not necessary for the pretameness of $\SSS$: It is easy to construct a non-stratified symmetric system $\SSS=\<\PP,\G\,\F>$ with $\PP$ pretame, such that $\SSS$ doesn't add new sets or classes in its symmetric extensions, hence clearly preserves $\GB^-$ (or also $\GB$ in case it holds in the ground model). For example, this can be achieved with the symmetric system that consists of the natural class product forcing to add a proper class of Cohen reals, with the permutation group $\G$ of arbitrary finitary permutations of indices, and with the trivial filter $\{\G\}$.} Unlike for the case of class forcing (where pretameness is a combinatorial requirement on class partial orders), we do not know of a combinatorial property of class symmetric systems that is equivalent to their pretameness.

\medskip

It will also be useful to know that when $\SSS$ is stratified, the sequence of predense sets witnessing combinatorial pretameness can be taken to be symmetric:

\begin{observation}\label{obs:sym}
If $\SSS$ is combinatorially pretame and stratified, then for every $p\in P$ and every symmetric sequence $\<D_i\mid i\in I>$ of dense subclasses of~$\PP$  with $I$ a set, there is $q\le p$ and a symmetric sequence $\<d_i\mid i\in I>\in V$ such that $d_i\subseteq D_i$ and $d_i$ is predense below $q$ whenever $i\in I$.
\end{observation}

\begin{proof}
Let $p\in P$, and let $\<D_i\mid i\in I>$ be such a sequence, say with $H\in\F$ such that $\pi[D_i]=D_{\pi(i)}$ whenever $\pi\in H$. Using combinatorial pretameness, we obtain $q\le p$ and a set $\<d_i\mid i\in I>$ such that $d_i\subseteq D_i$ and $d_i$ is predense below $q$ whenever $i\in I$. Using that $I$ is a set and $\SSS$ is stratified, let $b\subseteq P$ be a symmetric set such that $d_i\subseteq b$ for every $i\in I$. It follows that $\<D_i\cap b\mid i\in I>$ is a set, where $d_i\subseteq D_i\cap b\subseteq D_i$ and $D_i\cap b$ is predense below~$q$ whenever $i\in I$. But $\pi[D_i\cap b]=D_{\pi(i)}\cap b$ for any $\pi\in H\cap\sym(b)\in\F$, so $\<D_i\cap b\mid i\in I>$ is also a symmetric sequence.
\end{proof}

Finally, we want to show that we can have a notion of class forcing $\PP$ that is not pretame, while having a symmetric system $\SSS=\<\PP,\G,\F>$ which is pretame, but not combinatorially pretame.

\medskip

Let $\PP$ be the notion of class forcing to add a cofinal function from $\omega$ to the class of ordinals by proper initial segments, that is, conditions in $\PP$ are of the form $p\colon n\to\Ord$ for $n\in\omega$, ordered by end-extension. As shown in \cite[Section 2]{hklns}, this class forcing notion admits forcing relations, does not add new sets, clearly adds a cofinal function from $\omega$ to the class of ordinals, and is therefore not pretame.

Let $\G$ be the group of permutations $\pi=\<\pi_n\mid n<\omega>$ where, if $p\in\PP$, for every $n\in\dom(p)=\dom(\pi(p))$, $\pi(p)(n)=\pi_n(p(n))$. Let $\F=\{\{\G\}\}$.

\begin{claim}
$\SSS$ doesn't add new classes and therefore is pretame.
\end{claim}

\begin{proof}
Assume $\dot\Gamma$ is a class $\SSS$-name. Since $\PP$ doesn't add new sets, we can assume that $\dot\Gamma$ consists of elements of the form $\<\check x,p>$. Pick one such $\<\check x,p>\in\dot\Gamma$. Let $q\in\PP$ be arbitrary. Let $\pi\in\G$ be such that $\pi(p)\parallel q$. Then, also $\<\check x,\pi(p)>\in\dot\Gamma$. In total, $\dot\Gamma$ contains $\<\check x,r>$ for a predense class of conditions $r$ in $\PP$. This means however that $1\Sforces\check x\in\dot\Gamma$, yielding $\dot\Gamma$ to be forced to equal a class from the ground model.
\end{proof}

\begin{claim}
$\SSS$ is not combinatorially pretame.
\end{claim}
\begin{proof}
Let $I=\{\check i\mid i\in\omega\}$, and let $D_{\check i}=\{p\in\PP\mid\dom(p)\in[i,\omega)\}$ for every $i<\omega$; then $\<D_{\check i}\mid i\in\omega>$ is a symmetric sequence of dense subclasses of $\PP$. Let $q\in\PP$ be arbitrary, and assume that $\<d_i\mid i<\omega>$ is such that $d_i\subseteq D_i$ and $d_i$ is predense below $q$ whenever $i\in\omega$. Pick $i>\dom(q)$. If $r\le q$, then $r(i-1)$ can be an arbitrary ordinal, and different choices produce incompatible conditions. But $i-1\in\dom(r)$ whenever $r\in d_i$. This means that $d_i$ would have to be a proper class, which is a contradiction.
\end{proof}

\section{Preserving the powerset axiom}\label{section:powerset}

Following \cite{syhandbook}, we say that a notion of class forcing $\PP$ is \emph{tame} if it is pretame and preserves the powerset axiom. In complete analogy, we say that a class symmetric system $\SSS$ is \emph{tame} if it is pretame and preserves the powerset axiom.

\begin{theorem}\label{th:ptamestame}
If $\PP$ is tame and $\SSS=\<\PP,\G,\F>$ is a stratified class symmetric system, then $\SSS$ is tame.
\end{theorem}

\begin{proof}
We already know that $\SSS$ admits forcing relations~$\Sforces$ and preserves $\GB^-$. We also know that $\PP$ admits forcing relations $\Pforces$ and preserves $\GB$.

Let $p\in P$ and $\dot a\in\HS$. We want to find a condition below $p$ which forces in $\SSS$ that a set covering the powerset of $\dot a$ exists. Since $\PP$ is tame, there is $p_0\le p$ and $\dot b\in V^{\PP}$ such that $p_0\Pforces\dot b=\mathcal P(\dot a)$. Let $B=\dom(\dot b)$. Clearly, $p_0\Pforces\mathcal P(\dot a)\subseteq B^\bullet$. Now consider the $\PP$-name
\[
\dot R=\{\<\<\dot x,\dot y>^{\check{}},q>\mid\dot x\in V^{\PP}\land\dot y\in\HS\land q\in P\land(q\Pforces\dot x=\dot y\;\lor\;(\dot y=\check\emptyset\land q\Pforces\dot x\notin\HS^\bullet))\}.
\]
It is clear that $\1\Pforces(\forall x\in\check B)(\exists y)(\<x,y>\in\dot R)$, so since $\PP$ preserves Collection, there is $p_1\leq p_0$ and $\dot c\in V^{\PP}$ such that
\[
p_1\Pforces((\forall x\in \check B)(\exists y\in\dot c)(\<x,y>\in\dot R)\;\land\;\dot c\subseteq\check\HS).
\]
Let $\alpha$ be the name rank of $\dot c$. Then $\1\Pforces\check d\notin\dot c$ whenever the rank of $d$ is higher than $\alpha$, so $\1\Pforces\dot c\subseteq(V_\alpha\cap\HS)^{\check{}}$. Using stratification of $\SSS$, let $b\subseteq P$ be a symmetric set with $V_\alpha\cap P\subseteq b$, so certainly $V_\alpha\cap\HS\subseteq\HS_\alpha^b$. Then $\1\Pforces\dot c\subseteq(\HS_\alpha^b)^{\check{}}$, and our construction easily yields that
\[
p_1\Pforces B^\bullet\cap\HS^\bullet\subseteq(\HS_\alpha^b)^\bullet.
\]
Together with our earlier observation, this implies that
\[
p_1\Pforces\mathcal P(\dot a)\cap\HS^\bullet\subseteq(\HS_\alpha^b)^\bullet,
\]
or spelled out more completely, that
\[
p_1\Pforces(\forall x\in\HS^\bullet)(x\subseteq\dot a\,\to\,x\in(\HS_\alpha^b)^\bullet).
\]
Since the above formula that is forced by $p_1$ is the relativization to $\HS^\bullet$ of the formula $\mathcal P(\dot a)\subseteq(\HS_\alpha^b)^\bullet$, by \cref{lem:frs}, this means that $p_1\Sforces\mathcal P(\dot a)\subseteq(\HS_\alpha^b)^\bullet$.
\end{proof}

Let us show that the above does not generally hold without the stratification assumption.

\begin{proposition}
There is a tame notion of class forcing $\PP$ and a class symmetric system $\SSS=\<\PP,\G,\F>$ such that the powerset axiom is not preserved by $\SSS$. 
\end{proposition}

\begin{proof}
Let $\PP$ be the notion of class forcing in which conditions are finite partial functions $p$ from $\omega\times\omega\times\Ord$ to $2$, ordered by end-extension, with the additional property that whenever $\gamma,\gamma'\in\Ord$ and $\alpha,n\in\omega$ are such that both $\<\alpha,n,\gamma>$ and $\<\alpha,n,\gamma'>$ are in the domain of $p$, then \[p(\alpha,n,\gamma)=p(\alpha,n,\gamma').\]
In effect, $\PP$ just adds countably many Cohen reals, indexed by $\alpha<\omega$, so it is clearly tame.
Now, let $\G$ be the group of permutations $\pi=\<\pi_\alpha\mid\alpha<\omega>$ with each $\pi_\alpha$ a permutation of $\Ord$ moving only set-many ordinals. If $p\in\PP$, we let $\pi(p)(\alpha,n,\pi_\alpha(\gamma))=p(\alpha,n,\gamma)$
for all $\alpha,n<\omega$ and $\gamma\in\Ord$. Let $\F$ be the filter of subgroups of $\G$ given by subgroups of the form $\fix(e)$ for finite $e\subseteq\omega$, consisting of all permutations $\pi$ such that $\pi_\alpha=\id$ whenever $\alpha\in e$. As usual, if $\dot x\in\HS$ and $\fix(e)\subseteq\sym(\dot x)$, we call $e$ a support of $\dot x$. For every $\alpha<\omega$, the $\alpha^{\textrm{th}}$ Cohen real added by $\PP$ has a symmetric name $\dot c_\alpha$ with support $\{\alpha\}$. Assume for a contradiction that $\dot a\in\HS$ is a name for the powerset of $\omega$, as forced (with respect to $\SSS$) by a condition $p\in\PP$, and that $\dot a$ has finite support $e\subseteq\omega\times\Ord$. Let $\alpha\in\omega\setminus e$. Since $p\Sforces\dot c_\alpha\in\dot a$, there has to be a condition $q$ appearing in $\dot a$ which has nontrivial information on its $\alpha^{\textrm{th}}$ coordinate: otherwise, let $\QQ$ be the subforcing of $\PP$ consisting only of conditions with trivial information on their $\alpha^{\textrm{th}}$ coordinate. Then it is forced (by $\PP$) that $\dot a$ is contained already in the $\QQ$-generic extension, which however does not contain $\dot c_\alpha$, for $\dot c_\alpha$ is Cohen generic over this extension, and this is clearly a contradiction, as also $p\Pforces\dot c_\alpha\in\dot a$. But then, also $\pi(q)$ appears in~$\dot a$ whenever $\pi\in\fix(e)$. However $\{\pi(q)\mid\pi\in\fix(e)\}$ is a proper class, since for $\pi\in\fix(e)$, $\pi_\alpha$ can permute set-size subsets of $\Ord$ in arbitrary ways. This clearly contradicts that $\dot a\in\HS$ is set-sized.
\end{proof}

The class forcing notion $\PP$ that we use as a counterexample in the above proposition is antisymmetric, however not separative. It is easy to modify $\PP$ so that it still witnesses our above proposition and is separative, however not antisymmetric anymore. We do not know of a class forcing notion that witnesses the above proposition and is similarly antisymmetric and separative.

\section{Final remarks and open questions}

We end the paper with some observations and open questions. The following remains open:

\begin{question}
Is there a combinatorial characterization of the pretameness of class symmetric systems?
\end{question}

We would like to observe that there are important class symmetric systems that necessarily use class forcing notions which are not tame. In order to formalize this, we again work in an ambient universe $V$ with a transitive model $\MM$ of $\GB$. We may even assume, for the sake of simplicity, that $V\models\ZFC$ and $\MM$ is countable in~$V$.\footnote{Note that this assumption has non-trivial extra consistency strength. The assumption of a countable transitive model $\MM$ satisfying each axiom of $\ZFC$ individually, i.e., expressed as a scheme, does not increase consistency strength by the reflection theorem. Even in the context of $\GB$, a theory expressing that $\MM$ is an inner model of $\GB$ and stating the existence of a generic over~$\MM$ for a class forcing that does admit forcing relations is still equiconsistent with $\GB$ alone, since the forcing relation interprets this theory and $\GB$ proves that no contradiction is forced. However, it seems unclear how to conceive of a theory, of strength not exceeding $\GB$, that meaningfully formalizes constructing a class symmetric extension of a model of $\GB$ by use of a generic. This is unclear specifically in the case that the relevant class symmetric system admits forcing relations, but the underlying class forcing notion does not.}

\begin{definition}
We say that two class symmetric systems $\SSS$ and $\SSS'$ of $\MM$ are \emph{weakly equivalent} if they produce the same set of symmetric extensions (in $V$).
\end{definition}

Let us consider the question whether, if $\SSS=\<\PP,\G,\F>$ is a tame class symmetric system in~$\MM$, there is a weakly equivalent symmetric system $\SSS'=\<\PP',\G',\F'>$ in $\MM$ with a tame notion of class forcing $\PP$. This has a negative answer, letting $\SSS$ be the class symmetric system to obtain Gitik's model~\cite{gitik1980} in which every uncountable cardinal is singular. If a weakly equivalent $\SSS'$ with a tame notion of class forcing~$\PP'$ existed, let $G$ be $\PP'$-generic over $\MM$. Then, $\MM[G]$ has a submodel that is an $\SSS$-generic extension of $\MM$, in which every uncountable cardinal is singular. By upward absoluteness of this property, this is clearly a contradiction. Another example that does not use large cardinals is the Morris model (see \cite{Morris}, and \cite{Karagila2020} for a modern exposition).

\medskip

Regarding pretameness, the analogous question seems to be open:

\begin{question}
If $\SSS$ is pretame in a model of $\GB$, does this model have a weakly equivalent symmetric system $\SSS'=\<\PP,\G,\F>$ with $\PP$ pretame?
\end{question}

A positive answer to the following question would imply that stratification is, in a certain sense, essential for tame symmetric systems:

\begin{question}
If $\SSS$ is a tame symmetric system in a model of $\GB$, does this model have a weakly equivalent tame symmetric system $\SSS'$ that is stratified?\footnote{It is automatic that $\SSS'$ preserves $\GB$, but it doesn't seem to be automatic that forcing relations for $\SSS'$ exist. This is why we also require $\SSS'$ to be tame.}
\end{question}

Regarding tame notions of class forcing, the following variant of the above question imposes itself:

\begin{question}
If $\SSS=\<\PP,\G,\F>$ is tame and $\PP$ is tame, is there $\SSS'=\<\PP',\G',\F'>$ that is weakly equivalent to $\SSS$ and stratified with $\PP$ tame?\footnote{In this case, $\SSS'$ will be tame by \cref{th:ptamestame}.}
\end{question}

In the context of set-sized symmetric systems $\SSS=\<\PP,\G,\F>$, Grigorieff (see \cite{Grigorieff1975}, also \cite{symext}) showed that whenever $G$ is $\mathbb{P}$-generic over $\MM$, then $G$ is also (set-)generic over $\MMGS$. Trying to adapt the proof to class symmetric systems, one  should probably require $\SSS$ to admit forcing relations. The bigger obstacle however appears to be that when directly adapting Grigorieff's argument, one would obtain a forcing notion with conditions that are classes themselves. Therefore, the following question arises.

\begin{question}
Let $\SSS = \<\PP,\G,\F>$ be a class symmetric system in a model $\MM$ of $\GB$, admitting forcing relations. Let $G$ be $\PP$-generic over $\MM$. Can $G$ be added generically over $\MMGS$ via a notion of class forcing? 
\end{question}

\subsection*{Author Contribution}
The results of this paper have been achieved as a joint effort of its authors.

\subsection*{Acknowledgements}
This research was funded in whole or in part by the Austrian Science Fund (FWF) [10.55776/ESP5711024]; and the Mathematical Institute, University of Oxford. The authors would like to thank Asaf Karagila for some helpful discussions and remarks on the topics of this paper. For open access purposes, the authors have applied a CC BY public copyright license to any author-accepted manuscript version arising from this submission. No data are associated with this article.


\end{document}